\documentclass[11pt]{article}

\usepackage{etex}
\usepackage[margin={1.2in}]{geometry}
\usepackage[titletoc]{appendix}

\usepackage{rotating}

\usepackage{longtable}

\usepackage[sc]{mathpazo}
\usepackage{courier}

\usepackage[utf8]{inputenc}
\usepackage[T1]{fontenc}

\usepackage[english]{babel}
\usepackage{blindtext}

\usepackage{amsmath}

\usepackage{microtype}

\usepackage{amsmath}
\usepackage{amssymb}
\usepackage{amsthm}
\usepackage{color}
\usepackage{latexsym,extarrows}

\usepackage[all]{xy}

\usepackage[stable,multiple]{footmisc}

\RequirePackage[font=small,format=plain,labelfont=bf,textfont=it]{caption}
\numberwithin{equation}{section}

\begin{document}

\setlength{\unitlength}{1mm}
\def\e#1\e{\begin{equation}#1\end{equation}}
\def\ea#1\ea{\begin{align}#1\end{align}}
\def\eq#1{{\rm(\ref{#1})}}
\theoremstyle{plain}
\newtheorem{thm}{Theorem}[section]
\newtheorem{lem}[thm]{Lemma}
\newtheorem{prop}[thm]{Proposition}
\newtheorem{cor}[thm]{Corollary}
\theoremstyle{definition}
\newtheorem{dfn}[thm]{Definition}
\newtheorem{ex}[thm]{Example}
\newtheorem{rem}[thm]{Remark}
\newtheorem{conjecture}{Conjecture}


\def\q{\mathbf{q}}
\newtheorem*{mainconjecture}{{\bf Conjecture}}
\newtheorem*{mainthmA}{{\bf Theorem A}}
\newtheorem*{maincorC}{{\bf Corollary C}}
\newtheorem*{mainthmB}{{\bf Theorem B}}
\newtheorem*{mainthmD}{{\bf Theorem D}}

\newcommand{\LD}{\langle}
\newcommand{\RD}{\rangle}

\newcommand{\DL}{\langle\langle}
\newcommand{\DR}{\rangle\rangle}

\makeatletter
\newcommand{\subjclass}[2][2010]{%
  \let\@oldtitle\@title%
  \gdef\@title{\@oldtitle\footnotetext{#1 \emph{Mathematics Subject Classification.} #2}}%
}
\newcommand{\keywords}[1]{%
  \let\@@oldtitle\@title%
  \gdef\@title{\@@oldtitle\footnotetext{\emph{Key words and phrases.} #1.}}%
}
\makeatother

\title{\bf Ramanujan Identities and Quasi-Modularity in Gromov-Witten Theory}
\author{Yefeng Shen and Jie Zhou}
\date{}
\maketitle


\begin{abstract}
We prove that the ancestor Gromov-Witten correlation functions of one-dimensional compact Calabi-Yau orbifolds are quasi-modular forms. 
This includes the pillowcase orbifold which can not yet be handled by using Milanov-Ruan's B-model technique.
We first show that genus zero modularity is obtained from the phenomenon that the system of WDVV equations is essentially equivalent to the set of Ramanujan identities satisfied by the generators of the ring of quasi-modular forms for a certain modular group associated to the orbifold curve. Higher genus modularity then follows by using tautological relations. 
\end{abstract}


\setcounter{tocdepth}{2} \tableofcontents


\section{Introduction}

A folklore conjecture from physics \cite{Bershadsky:1993ta, Bershadsky:1993cx, Aganagic:2006wq} says that the generating functions of Gromov-Witten (GW) invariants of Calabi-Yau (CY) manifolds are quasi-modular forms or their generalizations. This is a remarkable conjecture since we know very little about higher genus GW invariants beyond saying that their generating functions are formal series. 

There are several important results in this direction.
For elliptic curves, the conjecture on the descendant GW invariants was solved in \cite{Dijkgraaf:1995, Kaneko:1995, Eskin:2001, Okunkov:2002}. 
For K3 surfaces, the generating functions of reduced GW invariants can be expressed in terms of modular forms by the KKV conjecture \cite{Katz:1999xq}. This was further studied in a sequence of works including
\cite{Yau:1996, Beauville:1999, Katz:1999xq, Bryan:2000, Klemm:2010, Maulik:2010, Maulik:2013, Pandharipande:2014}.
For CY 3-folds, which are among the most interesting targets, 
Huang-Klemm-Quackenbush \cite{Huang:2006hq} used the global properties of the generating functions to predict higher genus GW invariants of the quintic 3-fold up to genus 51. 
Recently, the modularity for higher genus GW theory for some special non-compact CY 3-folds was established in \cite{Aganagic:2006wq, Alim:2013eja}.
The results for CY 3-folds were obtained by using mirror symmetry which follows closely the original approach of BCOV \cite{Bershadsky:1993cx}.

The understanding of modularity proves to be both beautiful and useful in complete calculations of GW invariants. \\

A few years ago, Milanov and Ruan \cite{Milanov:2011} showed that such modularity properties can be extended to compact CY \emph{orbifolds}. They proved that three out of the four elliptic orbifold $\mathbb{P}^1$'s have modularity, by using mirror symmetry and the reconstruction techniques in \cite{Krawitz:2011}.
Their approach relies on the existence of a higher genus B-model mirror, which is a very difficult problem on its own.  

A well-known example is the elliptic orbifold curve $\mathbb{P}^1_{2,2,2,2}$ (also called the pillowcase orbifold \cite{Eskin:2006}), for which the  B-model techniques we know so far, including those in \cite{Milanov:2011}, do not apply.

In searching for a way to handle the example $\mathbb{P}^1_{2,2,2,2}$, we introduce a purely A-model approach which does not use ingredients from mirror symmetry. 
The approach combines and extends the ideas and techniques in \cite{Satake:2011} and \cite{Krawitz:2011}.
We start by observing that a subset of the WDVV equations, which are among the most fundamental relations in GW theory, coincides with the Ramanujan identities for the generators of the ring of quasi-modular forms for some modular group \cite{Satake:2011}. 
Based on this, it follows naturally that the genus zero theory is modular. Using Getzler's relation \cite{Getzler:1997} and
some tautological relations \cite{Ionel:2002, Faber:2005, Faber:2010}
on the moduli space of pointed curves, we 
reconstruct higher genera generating series from the genus zero ones, following \cite{Krawitz:2011},
and hence
obtain modularity for all genera. We then apply the same strategy to all one-dimensional compact CY orbifolds in a systematical way. 

Hence in addition to reproducing the all genera results about all elliptic orbifold curve except for the pillowcase orbifold in \cite{Milanov:2011, Milanov:2014} obtained by using mirror symmetry,
and the genus zero result about two of the elliptic orbifold curves ($\mathcal{X}_{r}, r=2,3$ below) in \cite{Satake:2011} obtained by using WDVV equations, we establish the modularity of all genera GW theory for $\mathbb{P}^1_{2,2,2,2}$, which could not be easily handled by A-model or B-model techniques previously.  

As a comparison to the overlapping results in \cite{Milanov:2011, Satake:2011, Milanov:2014}, the results we obtain are the same but the
viewpoint we take in this paper is slightly different. Namely, we treat the building blocks of the GW generating functions and the generators for the ring of quasi-modular forms as equally fundamental objects which get identified through the differential equations, that is, the WDVV equations and Ramanujan identities, that they satisfy.
This viewpoint is what makes it possible to 
isolate the small number of building blocks among the 
large number of terms involved in the potentials (cf. the genus zero potentials in the $\mathcal{X}_{r}, r=4,6$ cases that \cite{Satake:2011} could not fully determine) which are related by a complicated system of coupled differential equations.
It is also what leads to the later progress on the proof of LG/CY correspondence for these elliptic orbifold curves via the Cayley transform in representation theory \cite{Shen:2016}.

Although our strategy has only been carried out in dimension one cases so far, we expect that it can be generalized to higher dimensions.


\subsection{Gromov-Witten theory of Calabi-Yau 1-folds}

We now give a quick review of the GW theories of the elliptic orbifold curves.

An elliptic orbifold $\mathbb{P}^1$ is a compact complex orbifold, which is a quotient of an elliptic curve, with its  coarse moduli space the projective line $\mathbb{P}^1$. There are four such elliptic orbifolds, depicted in Figure \ref{figureellipticorbifolds}.
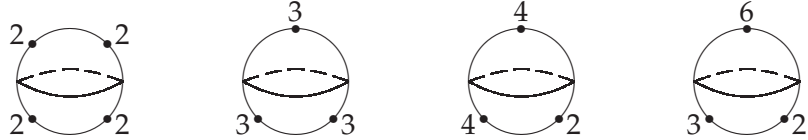
\begin{figure}[h]
   \renewcommand{\arraystretch}{1.2} 
\begin{displaymath}
\begin{picture}(100,25)


    \put(10,10){\circle{14}}

    \put(40,10){\circle{20}}

    \put(70,10){\circle{20}}

    \put(100,10){\circle{20}}

\multiput(0,0)(30,0){4}
{
\qbezier(3,10)(10,6)(17,10)
}

\multiput(0,0)(30,0){4}
{
\qbezier(3,10)(4,10.4)(5,10.7)
\qbezier(6,11)(7,11.3)(8,11.4)
\qbezier(9,11.6)(10,11.8)(11,11.6)
\qbezier(12,11.4)(13,11.3)(14,11)
\qbezier(15,10.7)(16,10.4)(17,10)
}

\put(5,5){\circle*{1}}
\put(15,5){\circle*{1}}
\put(5,15){\circle*{1}}
\put(15,15){\circle*{1}}
\put(2,3){2}
\put(16,3){2}
\put(2,15){2}
\put(16,15){2}

\put(35,5){\circle*{1}}
\put(40,17){\circle*{1}}
\put(45,5){\circle*{1}}
\put(39,18){3}
\put(32,3){3}
\put(46,3){3}

\put(65,5){\circle*{1}}
\put(70,17){\circle*{1}}
\put(75,5){\circle*{1}}
\put(69,18){4}
\put(62,3){4}
\put(76,3){2}

\put(95,5){\circle*{1}}
\put(100,17){\circle*{1}}
\put(105,5){\circle*{1}}
\put(99,18){6}
\put(92,3){3}
\put(106,3){2}

\end{picture}
\end{displaymath}
  \caption[Elliptic orbifold $\mathbb{P}^{1}$'s]{Elliptic orbifold curves $\mathcal{X}_{r}, r=2,3,4,6$.}
  \label{figureellipticorbifolds}
\end{figure}
We denote them by $\mathbb{P}^1_{2,2,2,2}, \mathbb{P}^1_{3,3,3}, \mathbb{P}^1_{4,4,2}$, and $\mathbb{P}^1_{6,3,2}$ respectively. The subscripts denote the orders of the cyclic isotropy groups of the non-trivial orbifold points.
These elliptic orbifold  $\mathbb{P}^1$'s, together with the elliptic curve, constitute all compact complex CY orbifolds\footnote{Here $X$ is CY if $c_1(TX)=0$, i.e, the first Chern class of the tangent bundle $TX$ is trivial.} of dimension one. Since these elliptic orbifolds are actually quotients of elliptic curves, we introduce the following convenient notation to denote all of the CY 1-folds:
\begin{equation*}
\mathcal{X}_r:=\mathcal{E}_r/\mathbb{Z}_r, \quad r=1, 2,3,4,6\,.
\end{equation*}
Here $r$ is the maximal order of the isotropy groups and $\mathcal{E}_r$ is some appropriate elliptic curve. For the case $r=1$, we use the convention $\mathbb{Z}_{1}=\{1\}$.
The K\"ahler cone of $\mathcal{X}_r$ is one-dimensional. It is generated by the Poincar\'e dual of the point class, denoted by $P\in H^{1,1}(\mathcal{X}_r,\mathbb{C})\cap H^2(\mathcal{X}_r,\mathbb{Z})$ and called the \emph{divisor class} or $P$-class throughout this paper.\\

Orbifold GW theory studies intersection theory on the moduli space of orbifold stable maps \cite{Chen:2002, Abramovic:2008}. 
More precisely, let $\overline{\mathcal{M}}_{g,n,d}^{\mathcal{X}_r}$ be the moduli space of orbifold stable maps from $n$-pointed stable orbifold curves of genus $g$ to $\mathcal{X}_r$, with degree $d=\int_{\beta}P\in\mathbb{Z}_{\geq 0}$,
with $\beta$ the image of the fundamental class. For the CY 1-folds, since $c_{1}(T\mathcal{X}_{r})=0$, the moduli space $\overline{\mathcal{M}}_{g,n,d}^{\mathcal{X}_r}$ has virtual dimension
\begin{equation}\label{dim-virt}
\mathrm{virdim}_{\mathbb C}=(3-\dim_{\mathbb C}\mathcal{X}_r)(g-1)+\int_{\beta}c_1(T\mathcal{X}_{r})+n\,
=2(g-1)+n\,.
\end{equation}

For $r=2,3,4,6$, let $H:= H_{\rm CR}^*(\mathcal{X}_r,\mathbb{C})$ be the \emph{Chen-Ruan comohology} of $\mathcal{X}_r$. 
For the elliptic curve $\mathcal{X}_1$, we let $H$ be its even cohomology.
The vector space $H$ is graded, with a non-degenerate pairing $\LD - ,- \RD$ given by the direct sum of the Poincar\'e pairings on the components of the \emph{inertia stack} $I\mathcal{X}_{r}$. It has a basis given by 
\begin{equation}\label{standard-basis}
\mathfrak{B}:=\left\{{\bf 1}, P\right\}\cup \mathfrak{B}_{\rm tw}:=\left\{{\bf 1}, P\right\}\cup \left\{\Delta_i^j\mid i=1,\cdots, m, j=1,\cdots a_i-1\right\}.
\end{equation}
Here $m$ is the number of orbifold points, $a_i$ is the order of the $i$-th orbifold point for $i=1,2,\cdots m$. We arrange the points in such a way that $a_1\geq a_2\cdots\geq a_m$. 
The element $\bf{1}$ is the identity in $H^{0}(\mathcal{X}_r,\mathbb{Z})$, and $\Delta_{i}$ is the Poincar\'e dual of the fundamental class of the $i$-th orbifold point for $i=1,2,\cdots, m$, with degree $2/a_i$. 
Moreover, $\Delta_{i}^j$ is \emph{ordinary cup product} of $j$-copies of $\Delta_i$. We call $\mathfrak{B}_{\rm tw}$ the \emph{set of twisted sectors}. For the elliptic curve, $\mathfrak{B}_{\rm tw}=\emptyset$.

The statement on modularity in the GW theory can be phrased intrinsically without using a basis, but for later use we shall choose and fix the aforementioned basis.

For $\{\phi_i\}_{i=1}^{n}\subset \mathfrak{B}$, one can define the \emph{ancestor GW invariants}
\begin{equation}\label{GW-correlator}
\LD \phi_1\psi_1^{k_1},\cdots,\phi_n\psi_n^{k_n}\RD_{g,n,d}:=\int_{[\overline{\mathcal{M}}_{g,n,d}^{\mathcal{X}_{r}}]^{\rm vir}}
\prod_{i=1}^{n}{\rm ev}_i^*(\phi_i) \prod_{i=1}^{n} \pi^*(\psi_i^{k_i})
\,.
\end{equation}
Here $\phi_{i},i=1,2,3\cdots n$ are called \emph{insertions}, $\psi_i := c_1(L_i)$ is called the $i$-th $\psi$-class, where $L_i$ is the $i$-th tautological line bundle on the moduli space $\overline{\mathcal{M}}_{g,n}$ of stable $n$-pointed curves of genus $g$. The map
$
{\rm ev}_{i}:\overline{\mathcal{M}}_{g,n,d}^{\mathcal{X}_{r}}\rightarrow I\mathcal{X}_{r}
$
is the evaluation map at the $i$-th point: it takes values in the \emph{inertia stack} $I\mathcal{X}_{r}$. While $\pi: \overline{\mathcal{M}}_{g,n,d}^{\mathcal{X}_{r}}\to \overline{\mathcal{M}}_{g,n}$ is the forgetful map, $[\overline{\mathcal{M}}_{g,n,d}^{\mathcal{X}_{r}}]^{\rm vir}$ is the virtual fundamental class of $\overline{\mathcal{M}}_{g,n,d}^{\mathcal{X}_{r}}$.

Let $t$ be the coordinate for the vector space spanned by the $P$-class and set $q=e^t$. Let $\Lambda_{g,n,d}^{\mathcal{X}_r}(\phi_1,\cdots,\phi_n):=\pi_*\left(\prod_{i=1}^{n}{\rm ev}_i^*(\phi_i)\right)\in H^{*}(\overline{\mathcal{M}}_{g,n})$. We define the \emph{GW class} 
$\Lambda_{g,n}^{\mathcal{X}_r}(\phi_1,\cdots,\phi_n):=\sum_{d\geq 0}\Lambda_{g,n,d}^{\mathcal{X}_r}(\phi_1,\cdots,\phi_n)\,q^d$ 
and the \emph{ancestor GW correlation function} by the following $q$-series
\begin{equation}\label{GW-correlation}
\DL \phi_1\psi_1^{k_1},\cdots,\phi_n\psi_n^{k_n}\DR_{g,n}:
=\int_{\overline{\mathcal{M}}_{g,n}}\Lambda_{g,n}^{\mathcal{X}_r}(\phi_1,\cdots,\phi_n)\prod_{i=1}^{n}\psi_i^{k_i}
=\sum_{d\geq0}\LD \phi_1\psi_1^{k_1},\cdots,\phi_n\psi_n^{k_n}\RD_{g,n,d}\,q^d.
\end{equation}
In particular, we can restrict to the cases where no $\psi$-class is included and call the corresponding quantities by 
\emph{primary
GW correlation functions}.
In this paper we shall occasionally call the GW correlation functions simply by \emph{correlators}.\\

One way to phrase the folklore conjecture for the CY 1-folds is as follows:
\begin{mainconjecture}\label{main-conj}
The GW correlation functions of CY 1-folds defined in \eqref{GW-correlation} are quasi-modular forms.
\end{mainconjecture}

By carefully analyzing the Givental-Teleman formula \cite{Givental:2001, Teleman:2012}, Milanov and Ruan \cite{Milanov:2011} proved a similar statement for \emph{ancestor correlation functions}, which are mirror \cite{Milanov:2011, Krawitz:2011, Satake:2011} to the ancestor GW correlation functions of elliptic orbifold $\mathcal{X}_r$ with $r=3,4,6$ respectively, from a global B-model of the universal unfoldings of simple elliptic singularities $E_6^{1,1}, E_7^{1,1},$ and $E_8^{1,1}$ \cite{Saito:1974}. 
So the above conjecture holds for the elliptic orbifolds $\mathcal{X}_{r}, r=3,4,6$. Unfortunately, their method fails for the elliptic orbifold $\mathcal{X}_2$ because such a (higher genus) mirror construction does not exist for $\mathcal{X}_2$.

A similar conjecture for \emph{descendant GW correlation functions} was proved for elliptic curves by \cite{Eskin:2001, Okunkov:2002}
which included the earlier results by
\cite{Dijkgraaf:1995, Kaneko:1995}
as a special case. Explicit formulas in terms of quasi-modular forms were given in \cite{Bloch:2000}.
In \cite{Eskin:2006}, Eskin and Okounkov proved that natural generating functions for enumeration of branched coverings of the pillowcase orbifold $\mathbb{P}^1_{2,2,2,2}$ are level 2 quasi-modular forms.  
Other various studies on orbifold GW theory and mirror symmetry of elliptic orbifolds can be found in, for example, \cite{Satake:2011, Krawitz:2011, Milanov:2011, Li:2011mi, Basalaev:2013, Hong:2014, Lau:2014}.\\

We will prove the conjecture for the ancestor GW correlation functions of the elliptic orbifolds in this paper.

Before we give the precise statements, we shall now briefly recall the basics on WDVV equations and Ramanujan identities. 


\subsection{WDVV equations for elliptic orbifold curves and Ramanujan identities}
\label{sectiongenuszeromodularity}

In GW theory, WDVV equations are equivalent to the associativity of \emph{quantum cup product} (or \emph{Chen-Ruan cup product} for orbifolds \cite{Chen:2004}). They give over-determined relations for GW invariants. In particular, we can translate some of them into differential equations satisfied by the corresponding correlators.
Recall that the divisor equation in GW theory allows us to get an equation relating the insertion of a divisor class $P$ to differentiation via $\theta_{q}:=q{\partial\over \partial q}$.
It schematically takes the form
\begin{equation}\label{theta-q}
\theta_q \DL\phi_1,\cdots,\phi_n\DR_{0,n}=\DL\phi_1,\cdots,\phi_n, P\DR_{0,n+1}\,.
\end{equation}
For elliptic orbifolds, we can rewrite the $P$-class as a Chen-Ruan cup product $\bullet$ of two twisted sectors ($P=r\Delta_i\bullet\Delta_i^{a_i-1}$), where $\bullet$ is defined via
\begin{equation}\label{Chen-Ruan}
\LD\phi_1\bullet\phi_2, \phi_3\RD=\DL\phi_1,\phi_2,\phi_3\DR_{0,3}\,.
\end{equation}
The right hand side of \eqref{theta-q} becomes a polynomial of simpler correlators after applying a suitable WDVV equation.

Ramanujan identities, on the other hand, arise in a completely different context and have a different nature.
They are equations satisfied by the generators of the ring of quasi-modular forms for subgroups of the full modular group $\mathrm{PSL}(2,\mathbb{Z})$.
These equations follow from the general theory of quasi-modular forms \cite{Kaneko:1995} which tells that for nice congruence subgroups the ring of quasi-modular forms is finitely generated
and is closed under the differential $\partial_{\tau}:={\partial\over \partial \tau}$, where $\tau$ is the coordinate on the upper half plane. 
For example, for the full modular group $\mathrm{PSL(2,\mathbb{Z})}$, the ring is generated by the familiar Eisenstein series\footnote{Here and below $\mathrm{Ei}_{2k},k\geq 1$ means the standard Eisenstein series of weight $2k$.} $\mathrm{Ei}_{2},\mathrm{Ei}_{4},\mathrm{Ei}_{6}$. They satisfy the following system of first order differential equations (ODEs)
\begin{equation}\label{eqRamanujanintro}
\begin{split}
{1\over 2\pi i}\partial_{\tau}\mathrm{Ei}_{2}&={1\over 12}(\mathrm{Ei}_{2}^{2}-\mathrm{Ei}_{4})\,,\\
{1\over 2\pi i}\partial_{\tau}\mathrm{Ei}_{4}&={1\over 3}(\mathrm{Ei}_{2}\mathrm{Ei}_{4}-\mathrm{Ei}_{6})\,,\\
{1\over 2\pi i}\partial_{\tau}\mathrm{Ei}_{6}&={1\over 2}(\mathrm{Ei}_{2}\mathrm{Ei}_{6}-\mathrm{Ei}_{4}^{2})\,.
\end{split}
\end{equation}
These identities were firstly found by Ramanujan and are termed \emph{Ramanujan identities}. In this paper, we shall also call by Ramanujan identities the set of differential equations satisfied by the generators for more general subgroups of $\mathrm{PSL}(2,\mathbb{Z})$.

We observe that these Ramanujan identities resemble the same structure of 
the WDVV equations obtained from using \eqref{theta-q} which relate derivatives of correlators to their polynomials.
This is the starting point of the present work.
In this paper, we carefully study the WDVV equations for all elliptic orbifold $\mathbb{P}^1$'s and get the following result.
\begin{thm}\label{thmequivalence}
For each of the elliptic orbifolds $\mathcal{X}_{r}$ with $r=2,3,4,6$, the system of WDVV equations satisfied by the basic genus zero orbifold GW correlation functions is equivalent to the set of Ramanujan identities for the corresponding modular group $\Gamma(r)$, the principal subgroup of level $r$ of the full modular group $\mathrm{PSL}(2,\mathbb{Z})$. 
\end{thm}
The Ramanujan identities for the modular group $\Gamma(r)$ will be reviewed in Section \ref{sectionmodular}. 
The outline of the proof of this theorem given in Section \ref{sectionequivalence} is 
as follows. WDVV equations provide a powerful tool in calculating primary GW invariants at genus zero.
Specializing to the elliptic orbifolds $\mathcal{X}_{r},r=2,3,4,6$, as shown in \cite{Satake:2011, Krawitz:2011}, the whole primary genus zero theory is determined by using WDVV equations and some other axioms in GW theory including the divisor equation, string equation, the dimension formula etc.
In particular, it was proved in \cite{Krawitz:2011} that all the primary correlators are reconstructed from a subset of very few correlators, called \emph{basic correlators}. A list of all the basic correlators and the WDVV equations 
for the reconstruction algorithm
were given explicitly in the Appendix of \cite{Shen:2013thesis}. Therefore, we only need to deal with the system of first-order differential equations (with $\theta_q$ as the derivative) satisfied by these basic correlators whose initial values can be worked out easily by direct calculations in GW theory. This system is then shown to be equivalent to the set of Ramanujan identities, in the sense explained in Section \ref{sectionequivalence}. 
 \\

Moreover, we check that the boundary conditions also match.
By using the existence and uniqueness theorems to systems of ordinary differential equations,
it follows automatically (see e.g. Section \ref{secequivalence2222}) that the corresponding correlation functions are quasi-modular forms.
Furthermore, by straightforward computations on the WDVV equations for \emph{non-basic correlators}, we obtain the following theorem.
\begin{thm}\label{thmprimarygenuszeromodularity}
Let $\mathcal{X}_r$ be an elliptic orbifold $\mathbb{P}^1$, $r=2,3,4,6$. Let $\phi_{i}\in\mathfrak{B}$, $i=1,\cdots, n$ be as in \eqref{standard-basis}. Then any primary GW correlation function $\DL \phi_1,\cdots,\phi_n\DR_{0,n}$ is a quasi-modular form for $\Gamma(r)$ with weight $T+2D-2$, where $T$ is the number of twisted sectors and $D$ is the number of $P$-classes among the insertions $\phi_1, \cdots, \phi_n.$
\end{thm}
The weight formula in the statement is a consequence of the WDVV equations as well. Later in this paper we shall refer to the above results as \emph{genus zero modularity}.
The proofs will be given in Section \ref{sectionhighergenusmodularity}.
Schematically, we will have the following correspondence.
\begin{table}[h]
  \centering
  \caption[Correspondence between WDVV and Ramanujan]{Correspondence between WDVV and Ramanujan}
  \label{tablecorrespondence}
  \renewcommand{\arraystretch}{1.2} 
 \begin{tabular}{c|c}
divisor axiom  &  $\theta_{q}$ derivative\\
WDVV equations &   Ramanujan identities\\
GW correlation functions & quasi-modular forms
\end{tabular}
\end{table}
That the modular group for the GW correlation function of $\mathcal{X}_{r}$
is the principal subgroup $\Gamma(r)$ of $\mathrm{PSL}(2,\mathbb{Z})$ is easier to see from mirror symmetry \cite{Milanov:2011, Milanov:2012, Milanov:2012qu}.
In the A-model approach that we are following this is proved from the machinery of modular forms.
 

\subsection{Tautological relations and modularity at higher genus}

WDVV equations are the simplest tautological relations over $\overline{\mathcal{M}}_{g,n}$. Other tautological relations are found, including the most general results, in \cite{Pandharipande:2015}.
Tautological relations have been very successfully used to compute GW invariants at higher genus. 
In this paper, we shall use both Getzler's relation \cite{Getzler:1997} over $\overline{\mathcal{M}}_{1,4}$ and the $g$-reduction technique \cite{Ionel:2002, Faber:2005} to prove the modularity for higher genus correlation functions.

By choosing appropriate insertions and applying Getzler's relation, we can express $\DL P\DR_{1,1}^{\mathcal{X}_r}$ in terms of primary genus zero correlators. This calculation shows $\DL P\DR_{1,1}^{\mathcal{X}_r}$ is a quasi-modular form for $\Gamma(r)$.
The $g$-reduction technique then allows us to express a higher genus correlation function as a polynomial of correlation functions at lower genera. That is, the lower genera correlation functions are the building blocks to form higher genera ones. 
As a consequence, we show that the ancestor GW correlation functions are in the polynomial ring of primary genus zero correlators, $\DL P\DR_{1,1}^{\mathcal{X}_r}$ and their $\theta_q$-derivatives which are still in the polynomial ring (according to the results in Section \ref{sectionRamanujanidentitiesandboundaryconditions} and Remark \ref{remarkpolynomiality}). 
Thus we arrive at the following conclusion whose detailed proof is given in Section \ref{sectionhighergenusmodularity}.

\begin{thm}\label{thmhighergenusmodularity}
Let $\mathcal{X}_r$ be a compact CY $1$-fold, $r=1,2,3,4,6$. Let $\phi_{i}\in\mathfrak{B}$, $i=1,\cdots, n$ be as in \eqref{standard-basis}.
Then any ancestor GW correlation function $\DL \phi_1\psi_1^{k_{1}},\cdots,\phi_n\psi_n^{k_{n}}\DR_{g,n}^{\mathcal{X}_r}$ is a quasi-modular form of $\Gamma(r)$ with weight 
$T+2D+2g-2,$ where $T$ is the number of twisted sectors and $D$ is the number of $P$-classes among $\phi_1, \cdots, \phi_n$.
\end{thm}
Note that as the theorem states, the results for the true orbifold cases also hold for the the elliptic curves, namely the $r=1$ case.
This is because the $g$-reduction technique also applies to this case.


\subsection{Future directions}

Modularity of the correlation functions would help simplifying significantly the calculations for the orbifold GW invariants. It also makes it easier to study the arithmetic aspects of the generating functions. For example, for the current cases, a certain version of integrality for the GW invariants can be obtained for free since this is so for the generators of the rings of quasi-modular forms.

The modularity has some other far-reaching consequences for the study of GW theory of these elliptic orbifolds. 
The originally locally defined generating series now become quasi-modular forms and hence (their non-holomorphic completions \cite{Kaneko:1995}) automatically extend to the whole moduli space.
This then might shed some light on exploring the global behavior of the correlation functions and on discussing some new aspects in enumerative geometry. 

In particular, we can make use of the modularity to analytically continue the correlation functions to other patches on the moduli space and to prove, as a toy model of, the LG/CY correspondence \cite{Krawitz:2011} for the elliptic orbifolds.
They might also give some hints in finding/checking a set of equations relating recursively correlation functions of different genera, called holomorphic anomaly equations \cite{Bershadsky:1993ta, Bershadsky:1993cx} and proposed in \cite{Milanov:2012} for the case of elliptic orbifolds.
It is hopeful that these equations can be used to compute the correlation functions in an even simpler way and to study finer structures of the sequence of correlation functions.


\subsection*{Outline of the paper}
The structure of this paper is as follows. In Section \ref{sectionmodular}
 we shall introduce some basics on modular groups and quasi-modular forms. To make the paper complete and self-contained we also list the generators for the rings of quasi-modular forms for the groups $\Gamma_{0}(N),N=1^{*},2,3,4$ which are used to construct the rings for $\Gamma(r),r=6,4,3,2$. 
The expressions for the generators of the ring of modular forms in terms of $\theta$-functions are collected in Appendix \ref{appendixthetaexpansions}.
In Section \ref{sectionequivalence}, we show the equivalence between the set of WDVV equations for $\mathcal{X}_{r},r=2,3,4,6$
and the Ramanujan identities among the generators for the ring of quasi-modular forms for $\Gamma(r)$.
The proof is based on straightforward computations.
In Section \ref{sectionhighergenusmodularity}, we prove the modularity for higher genera correlation functions for the elliptic orbifolds $\mathcal{X}_{r},r=1,2,3,4,6$. 


\subsection*{Acknowledgement}
We thank Murad Alim, Siu-Cheong Lau, Todor Milanov, Yongbin Ruan, Emanuel Scheidegger, Hsian-Hua Tseng, Shing-Tung Yau and Don Zagier for enlightening discussions on various aspects of elliptic orbifolds and modular forms, and particularly Yongbin Ruan for very helpful advice during different stages of this work.
Part of the work was done when Y. Shen was a Project Researcher at Kavli IPMU and J. Zhou was a graduate student at the mathematics department at Harvard.
We would like to thank both institutions for providing an excellent research atmosphere. 

Y. Shen is partially supported by NSF grant DMS-1159156. J. Zhou is supported
by the Perimeter Institute for Theoretical Physics.  Research at Perimeter Institute
is  supported  by  the  Government  of  Canada  through  Industry  Canada  and  by  the
Province of Ontario through the Ministry of Economic Development and Innovation.


\section{Quasi-modular forms and Ramanujan identities}

\label{sectionmodular}


\subsection{Modular groups, modular forms and quasi-modular forms}
\label{sectionmodularforms}

In this section we review the basics on modular groups and modular forms, mainly following \cite{Rankin:1977ab, Zagier:2008}.
The modular groups that are involved in our study include
the congruence
subgroups called \emph{Hecke subgroups} of
$\Gamma(1)=\mathrm{PSL}(2,\mathbb{Z})=\mathrm{SL}(2,\mathbb{Z})
/\{\pm I\}$:
\begin{equation}
\Gamma_{0}(N)=\left\{ \left.
\begin{pmatrix}
a & b  \\
c & d
\end{pmatrix}
\right\vert\, c\equiv 0\,~ \textrm{mod} \,~ N\right\}< \Gamma(1)\,.
\end{equation}
We will also consider the congruence subgroups called \emph{principal modular groups of level $N$}
\begin{equation}
\Gamma(N)=\left\{ \left.
\begin{pmatrix}
a & b  \\
c & d
\end{pmatrix}
\right\vert\,
\begin{pmatrix}
a & b  \\
c & d
\end{pmatrix}
\equiv
\begin{pmatrix}
1 & 0 \\
0 & 1
\end{pmatrix}
\,~ \textrm{mod} \,~ N\right\}< \Gamma(1)\,.
\end{equation}

Let $\mathcal{H}:=\{\tau\in\mathbb{C} \mid {\rm Im} \tau>0\}$ be the upper half plane. \emph{A modular form of weight $k\in\mathbb{Z}_{\geq 0}$ for the congruence subgroup $\Gamma$ of $\mathrm{PSL}(2,\mathbb{Z})$} is a
function $f : \mathcal{H}\rightarrow \mathbb{C}$ 
satisfying the following conditions:
\begin{itemize}
\item $\,f(\gamma \tau)=j_{\gamma}(\tau)^{k}f(\tau), \quad \forall \gamma\in \Gamma\,$, where $j$ is the \emph{automorphy factor} defined by
$$j: \Gamma \times \mathcal{H}\rightarrow
\mathbb{C},\quad \left(\gamma=\left(
\begin{array}{cc}
a & b  \\
c & d
\end{array}
\right),\tau\right)\mapsto j_{\gamma}(\tau):=(c\tau+d)\,.$$
\item $\,f$ is holomorphic on $\mathcal{H}$.
\item $\,f$ is \emph{holomorphic at the cusps}, in the sense that the function
\begin{equation}
\label{eqslashoperator} 
f|_{\gamma}: \tau\mapsto j_{\gamma}(\tau)^{-k}
f(\gamma\tau)
\end{equation}
is holomorphic at $\tau=i\infty$ for any $\gamma\in \Gamma(1)$.
\end{itemize}
The second and third conditions in the above can be equivalently
described as $f$ is holomorphic on the modular curve
$X_{\Gamma}=\Gamma\backslash \mathcal{H}^{*}$, where
$\mathcal{H}^{*}=\mathcal{H}\cup \mathbb{P}^{1}(\mathbb{Q})$, i.e.,
$\mathcal{H}\cup \mathbb{Q}\cup \{i\infty\}$. The first condition means that  $f$ can be formulated as a holomorphic section of a line bundle over $X_{\Gamma}$ whose transition function is defined by $j_{\gamma}^{k}$.

We can also define modular forms with multiplier system of integral
weight $k$ for $\Gamma$ by replacing the automorphy factor in
\eqref{eqslashoperator} by $j_\gamma(\tau)^{k} = \hat{\chi}(\gamma)(c\tau+d)^{k}$,
where $\hat{\chi}(\gamma)=(-1)^{k}\chi(\gamma)$ for $\gamma\in
\Gamma$ and $\chi:\Gamma\to\mathbb{C}^{*}$ is a multiplier system, see for example
\cite{Rankin:1977ab} for details. The space of modular forms with
multiplier system $\chi$ for $\Gamma$ forms a graded 
ring (the grading by the modular weight) and is denoted by $M_{*}(\Gamma,\chi)$. When $\chi$ is trivial,
we shall often omit it and simply write $M_{*}(\Gamma)$.

\begin{ex}
Taking the group $\Gamma$ to be the congruence subgroup $\Gamma(2)$, then the ring of even weight modular forms $M_{\mathrm{even}}(\Gamma(2))$ is generated by any two of the
three $\theta$-constants $\theta_{3}^{4}(\tau),\theta_{4}^{4}(\tau),\theta_{2}^{4}(\tau)$ which satisfies the relation
\begin{equation}\label{eqthetaidentities}
\theta_{3}^{4}(\tau)=\theta_{4}^{4}(\tau)+\theta_{2}^{4}(\tau)\,.
\end{equation}
Throughout this paper we follow the convention in \cite{Zagier:2008} for the $\theta$-constants which is reviewed in Appendix \ref{appendixthetaexpansions}.
\end{ex}

\emph{A quasi-modular form of weight $k$ for the
group $\Gamma$} is a function $f : \mathcal{H}\rightarrow
\mathbb{P}^{1}$ satisfying the second and third conditions above,
while with the first condition replaced by the following:
\begin{itemize}
\item $\,$ There exist holomorphic functions $f_{i}, i=0,1,2,3,\dots, k-1$ such that
\begin{equation}
f(
\gamma\tau)=j_{\gamma}(\tau)^{k}f(\tau)+\sum_{i=0}^{k-1}c^{k-i}j_{\gamma}(\tau)^i
f_{i}(\tau)\,,\quad \forall \gamma=\left(
\begin{array}{cc}
a & b  \\
c & d
\end{array}
\right)\in \Gamma\,.
\end{equation}
\end{itemize}
We denote the space of quasi-modular forms for $\Gamma$ by $\widetilde{M}_{*}(\Gamma)$. It is a graded
differential ring.
According to \cite{Kaneko:1995}, one has the following
structure theorem: 
\begin{equation}\label{eqstructurethmE2}
\widetilde{M}_{*}(\Gamma)=M_{*}(\Gamma)\otimes \mathbb{C}[\mathrm{Ei}_{2}]\,,
\end{equation}
where $\mathrm{Ei}_{2}$ is the usual weight two Eisenstein series given by 
\begin{equation*}
\mathrm{Ei}_{2}(\tau)=1-24\sum_{n=1}^{\infty}\sum_{d: d|n} d \exp (2\pi i n\tau)\,.
\end{equation*}
The derivative of a modular form is in general {\bf NOT} a modular form, but a quasi-modular form, as can be easily seen from the definitions.
Hence the ring $M_{*}(\Gamma,\chi)$ is not closed under the derivative $\partial_{\tau}:={\partial\over \partial \tau}$.
However, one can show that $\widetilde{M}_{*}(\Gamma,\chi)$ is so.

\begin{ex}
For the full modular group
$\Gamma(1)=\mathrm{PSL}(2,\mathbb{Z})$, we have
$M_{*}(\Gamma(1))=\mathbb{C}[\mathrm{Ei}_{4},\mathrm{Ei}_{6}]$, $\widetilde{M}_{*}(\Gamma(1))=\mathbb{C}[\mathrm{Ei}_{2},\mathrm{Ei}_{4},\mathrm{Ei}_{6}]$, where
$\mathrm{Ei}_{4}$ and $\mathrm{Ei}_{6}$ are the Eisenstein series given by
\begin{equation*}
\mathrm{Ei}_{4}(\tau)=1+240\sum_{n=1}^{\infty}\sum_{d: d|n} d^3 \exp (2\pi i n\tau)\,,\quad \mathrm{Ei}_{6}(\tau)=1-504\sum_{n=1}^{\infty}\sum_{d: d|n} d^5 \exp (2\pi i n\tau)\,.
\end{equation*}
The differential structure of the ring  $\widetilde{M}_{*}(\Gamma(1))$ is given by the Ramanujan identities in
\eqref{eqRamanujanintro}.
For a congruence subgroup $\Gamma<\Gamma(1)$, the quasi-modular forms satisfy similar equations,
which we shall study in Section \ref{sectionringofquasimodularforms}.
\end{ex}


\subsection{Quasi-modular forms for $\Gamma_0(N), N=1^{*},2,3,4$}
\label{sectionringofquasimodularforms}

In this section we shall first review the construction of differential rings of quasi-modular forms.
For reference and for self-containedness we list in this section the results which are scattered in the literature.
The material is largely taken from the expository part in \cite{Alim:2013eja}. 

We consider modular forms (with possibly non-trivial multiplier
systems) for the Hecke subgroups $\Gamma_{0}(N)$ with $N=2,3,4$ and
the subgroup $\Gamma_{0}(1^{*})$, which denotes the unique index $2$
normal subgroup of $\Gamma(1)=\mathrm{PSL}(2,\mathbb{Z})$ and whose (formal) Hauptmodul $\alpha$ is defined to be such that the $j$-invariant is given by $j(\alpha)={432/\alpha(1-\alpha)}$, see \cite{Maier:2009} for details on this. 
Here a Hauptmodul is a generator for the
rational function field of the genus zero modular curve.
All of them are of
genus zero in the sense that the corresponding modular curves
$X_{0}(N):=\Gamma_{0}(N)\backslash \mathcal{H}^{*}$ can be equipped with complex analytic structures as genus zero
Riemann surfaces. Each of the corresponding modular curve
$X_{\Gamma}$ has three singular points: two (equivalence classes) of
cusps\footnote{Here we use the notation $[\tau]$ to denote the $\Gamma$-equivalence class of $\tau\in \mathcal{H}^{*}$.} $[i\infty],[0]$, and the third one is a cusp or an
elliptic point, depending on the modular group. It is a quadratic
elliptic point $[\tau]=[(1+i)/2]$ for $N=2$, cubic elliptic point $[\tau]=[-i(\exp
2\pi i/3)/\sqrt{3}]$ for $N=3$ and $N=1^{*}$, and a cusp $[\tau]=[1/2]$ for $N=4$.
For a review of these facts, see for instance \cite{Rankin:1977ab}.

We can choose a particular Hauptmodul
$\alpha(\tau)$ for the corresponding modular group such that the two
cusps are given by $\alpha=0,1$ respectively, and the third singular point is
$\alpha= \infty$. It is given by 
\begin{equation}
\alpha(\tau)=C_{N}^{r}(\tau)/A_{N}^{r}(\tau)\,,
\end{equation}
where $r=6,4,3,2$ for the cases $N=1^{*},2,3,4$
respectively: the functions 
\begin{equation*}
A_{N}(\tau),\quad C_{N}(\tau)=\alpha(\tau)^{1\over
r}A(\tau), \quad B_{N}(\tau):=(1-\alpha(\tau))^{1\over r}A_{N}(\tau)
\end{equation*}
are given in
Table \ref{tableetaexpansions}. 
\begin{table}[h]
  \centering
  \caption[Expressions of $A_{N},B_{N},C_{N}$ for $\Gamma_{0}(N), N=1^{*},2,3,4$]{Expressions of $A_{N},B_{N},C_{N}$ for $\Gamma_{0}(N), N=1^{*},2,3,4$}
  \label{tableetaexpansions}
  \renewcommand{\arraystretch}{1.2} 
 \begin{tabular}{c|ccc}
$N$&$A_{N}$&$B_{N}$&$C_{N}$\\
\hline
$1^{*}$&$E_{4}(\tau)^{1\over 4}$&$({E_{4}(\tau)^{3\over 2}+E_{6}(\tau)\over 2})^{1\over 6}$&$({E_{4}(\tau)^{3\over 2}-E_{6}(\tau)\over 2})^{1\over 6}$\\
$2$&${(2^{6}\eta(2\tau)^{24}+\eta(\tau)^{24} )^{1\over 4} \over \eta(\tau)^2\eta(2\tau)^2}$&${\eta(\tau)^{4}\over \eta(2\tau)^{2}}$&$2^{3 \over 2}{\eta(2\tau)^4\over \eta(\tau)^2}$ \\
$3$&${(3^{3}\eta(3\tau)^{12}+\eta(\tau)^{12} )^{1\over 3} \over \eta(\tau)\eta(3\tau)}$&${\eta(\tau)^{3}\over \eta(3\tau)}$&$3{\eta(3\tau)^3\over \eta(\tau)}$ \\
$4$&${(2^{4}\eta(4\tau)^{8}+\eta(\tau)^{8} )^{1\over 2} \over
\eta(2\tau)^2}=
{\eta(2\tau)^{10}\over\eta(\tau)^{4}\eta(4\tau)^{4}}$&${\eta(\tau)^{4}\over
\eta(2\tau)^2}$&$2^2{\eta(4\tau)^4\over \eta(2\tau)^2}$
\end{tabular}
\end{table}
Their expressions in terms of $\theta$-functions are listed in Appendix \ref{appendixthetaexpansions} from where one can easily see that the coefficients in the $q$-expansions are integral.
By definition, one has
\begin{equation}\label{ABCrelation}
A_{N}^{r}(\tau)=B_{N}^{r}(\tau)+C_{N}^{r}(\tau)\,.
\end{equation}
See \cite{Borwein:1991,
Berndt:1995} and also \cite{Maier:2009, Maier:2011} for a more detailed review on the modular forms $A_{N},B_{N},C_{N}$.
We shall drop the subscript $N$ in the notations when it is clear from the surrounding texts.

Moreover, for $N=2,3,4$ ($N=1^{*}$ case is exceptional) one has
\begin{equation}\label{eqGaussSchwarz}
A_{N}^{2}(\tau)
={1\over N-1}(N\mathrm{Ei}_{2}(N\tau)-\mathrm{Ei}_{2}(\tau))\,.
\end{equation}
Another useful observation is the following
\begin{equation}\label{C2relation}
B_{2}(\tau)=B_{4}(\tau),\quad C_{2}^{2}(\tau)=2A_{4}(\tau)C_{4}(\tau)\,.
\end{equation}
Again we denote the graded ring of modular forms with character $\chi$ for
$\Gamma$ by $M_{*}(\Gamma,\chi)$ and correspondingly the graded ring of
even weight modular forms by $M_{\mathrm{even}}(\Gamma,\chi)$, then we have the
following results
\begin{equation}\label{eqringofmodularformsgamma0N}
\begin{split}
M_{\mathrm{even}}(\Gamma_{0}(2))&=\mathbb{C}[A_{2}^{2},B_{2}^{4}]\,,\\
M_{*}(\Gamma_{0}(3),\chi_{-3})&=\mathbb{C}[A_{3},B_{3}^{3}]\,, \\
M_{*}(\Gamma_{0}(4),\chi_{-4})&=\mathbb{C}[A_{4},B_{4}^{2}]\,.
\end{split}
\end{equation}
Here $\chi_{-3}(d)=\left(\frac{-3}{d}\right)$ is the Legendre symbol
and it gives a non-trivial Dirichlet character for the modular
forms. Similarly, $\chi_{-4}(d)=\left(\frac{-4}{d}\right)$. Furthermore, we have
\begin{equation}\label{eqringofmodularformsgammaN}
\begin{split}
M_{\mathrm{even}}(\Gamma(2))&=\mathbb{C}[\theta_{3}^{4},\theta_{4}^{4}, \theta_{2}^{4}]/(\theta_{3}^{4}=\theta_{4}^{4}+\theta_{2}^{4})\,,\\
M_{*}(\Gamma(3))&=\mathbb{C}[A_{3},C_{3}]\,,\\
M_{*}(\Gamma(4))&=\mathbb{C}[A_{4},C_{4},C_{2}]/(C_{2}^{2}=2A_{4}C_{4})\,.
\end{split}
\end{equation}
See \cite{Bannai:2001, Sebbar:2002, Maier:2011} and references therein for
details of these results.

We define further the quantity
\begin{equation}\label{defofE}
E_{N}(\tau)={1\over 2\pi i}\partial_{\tau}\log \left( C_{N}^{r}(\tau)B_{N}^{r}(\tau)\right)\,.
\end{equation}
It follows from the expressions for $B_{N},C_{N}$ in Table \ref{tableetaexpansions} that
\begin{eqnarray}
E_{N}(\tau)&=&{N\mathrm{Ei}_{2}(N\tau)+\mathrm{Ei}_{2}(\tau)\over N+1}\,, \quad  N=1^{*},2,3\,,\nonumber\\
E_{N}(\tau)&=&{\mathrm{Ei}_{2}(\tau)-2\mathrm{Ei}_{2}(2\tau)+4\mathrm{Ei}_{2}(4\tau)\over 3}\,, \quad  N=4\,.\label{eqEintermsofE2}
\end{eqnarray}
According to the transformations of $B_{N}^{r},C_{N}^{r}$ under the group $\Gamma_{0}(N)$, we can see that $E_{N}$ is a quasi-modular form for $\Gamma_{0}(N)$.
We can then replace the Eisenstein series in \eqref{eqstructurethmE2} by the quasi-modular form $E_{N}$. The statement still holds. That is, we have
\begin{equation}\label{eqstructurethm}
\widetilde{M}_{*}(\Gamma_{0}(N),\chi)=M_{*}(\Gamma_{0}(N),\chi)\otimes \mathbb{C}[\mathrm{Ei}_{2}]\cong M_{*}(\Gamma_{0}(N),\chi)\otimes \mathbb{C}[E_{N}]\,.
\end{equation}


\subsection{Ramanujan identities}
\label{sectionRamanujanidentitiesandboundaryconditions}

The ring generated by $A_{N},B_{N},C_{N},E_{N}$ serves as the largest ring when considering rings of quasi-modular forms for $\Gamma_{0}(N)$ and for $\Gamma(N)$ in our discussions.  
It is closed
under the derivative $\partial_{\tau}$. The differential structure is given by the following equations which can be easily derived. See e.g., \cite{Alim:2013eja, Zhou:2013hpa} and references therein for more details.
\begin{prop}\label{propRamanujanforall}
For each of the modular groups $\Gamma_{0}(N), N=1^{*},2,3,4$ with $r=6,4,3,2$ respectively, the
following identities hold:
\begin{eqnarray*}
{1\over 2\pi i}\partial_{\tau}A_{N}&=&{1\over 2r}A_{N}(E_{N}+{C_{N}^{r}-B_{N}^{r}\over A_{N}^{r}}A_{N}^{2})\,,\\
{1\over 2\pi i}\partial_{\tau}B_{N}&=&{1\over 2r}B_{N}(E_{N}-A_{N}^{2})\,,\\
{1\over 2\pi i}\partial_{\tau}C_{N}&=&{1\over 2r}C_{N}(E_{N}+A_{N}^{2})\,,\\
{1\over 2\pi i}\partial_{\tau}E_{N}&=&{1\over 2r}(E_{N}^{2}-A_{N}^{4})\,.
\end{eqnarray*}
\end{prop}
The asymptotic behavior of the quasi-modular forms will be useful later so we list them here ($Q=\exp (2\pi i \tau)$)
\begin{equation}\label{boundaryconditionsforGamma0N}
\begin{split}
A_{N}(Q)&=1+\mathcal{O}(Q)\,,\\
B_{N}(Q)&=1+\mathcal{O}(Q)\,,\\
C_{N}(Q)&=\kappa_{N}^{1\over r}Q^{1\over r}(1+\mathcal{O}(Q))\,,\\
E_{N}(Q)&=1+\mathcal{O}(Q)\,.
\end{split}
\end{equation}
The numbers $\kappa_{N}$, as well as the relation between $N$ and $r$ are given
in
Table \ref{tablearithmeticvalues}.
\begin{table}[h]
  \caption[Arithmetic numbers for $N=1^{*},2,3,4$]{Arithmetic numbers for $N=1^{*},2,3,4$}
  \label{tablearithmeticvalues}
  \renewcommand{\arraystretch}{1.2} 
\begin{displaymath}
\begin{tabular}{c|cccc}
 		 $N$ & 4 & 3 & 2 & $1^{*}$ \\
\hline
    $r$ & 2 & 3 & 4  & 6\\
    $\kappa_{N}$  & 16 & 27 & 64  & 432 \\
  \end{tabular}
	\end{displaymath}
\end{table}
The number $r$ is related to the signature $\nu$ (that is, its index as a subgroup of the full modular group) of the modular group $\Gamma_{0}(N)$ by $r=12/\nu$.

The equations in Proposition \ref{propRamanujanforall} are similar to the Ramanujan identities satisfied by the generators of the ring of quasi-modular forms for the full modular group $\Gamma(1)$, which were discussed earlier in \eqref{eqRamanujanintro}.
Hence we call these the \emph{Ramanujan identities for} $\Gamma_{0}(N),N=1^{*},2,3,4$.
Due to the relation in \eqref{ABCrelation} we can see that the system is actually redundant.

\begin{rem}\label{remellipticorbifoldandsingularity}
As mentioned earlier in Introduction, the mirror of the elliptic orbifolds $\mathcal{X}_{r}, r=3,4,6$ are, see  \cite{Milanov:2011}, the so-called simple elliptic singularities $E_{n}^{1,1},n=6,7,8$ \cite{Saito:1974}.
The corresponding elliptic curve families, up to base changes, are parametrized by the modular curves $\Gamma_{0}(N)\backslash \mathcal{H}^{*}$ with $N=3,2,1^{*}$, respectively.
On the other hand, for the $N=4$ case, the corresponding elliptic curve family is realized as a complete intersection in $\mathbb{P}^{3}$ and is parametrized by the modular curve $\Gamma_{0}(4)$.
See \cite{Milanov:2012qu, Alim:2013eja, Milanov:2014, Zhou:2014thesis} for more details on the arithmetic aspects of these families and their applications in GW theory and mirror symmetry. 
For the discussions in this paper, we are not going to use the results for the modular group $\Gamma_{0}(1^{*})$.
The reason that we include them here is for comparison and for the purpose of relating our results on modularity to those obtained from B-model considerations. This will be addressed further in Section \ref{sectonmodularitysummarized}.
\end{rem}


\subsection{Relations among modular forms induced by isogenies}
\label{sectioncongruenceforall}

In this work, when we talk about modularity, we shall switch to the variable $\tau$ or $Q$ with the $j$-invariant for an elliptic curve given by
$j(Q)={1/ Q}+744+\cdots$.
The previously defined parameter $q=e^t$ in orbifold GW theory is related to the parameter $Q=\exp (2\pi i\tau)$
by $q=Q^{1\over r}$, see \cite{Milanov:2011}. To express the correlation functions which are $q$-series in terms of $Q$, we need to apply the transformation
$ \tau\mapsto t/(2\pi i)=\tau/r$ which is called the \emph{$r$-isogeny}.

In this section, we summarize the action of the $r$-isogeny on the quasi-modular forms 
for the cases with $(N,r)=(4,2), (3,3), (2,4)$.
These results will be needed later in Section \ref{sectionequivalence}. The $r=6$ case is exceptional and we will deal with it separately later in Section \ref{sectionr=6}.

\paragraph{$(N,r)=(4,2)$.}

We have the following quadratic identities:
\begin{equation}
\theta_{3}^{2}(2\tau)={1\over 2}(\theta_{3}^{2}(\tau)+\theta_{4}^{2}(\tau))\,,\quad 
\theta_{4}^{2}(2\tau)=\theta_{3}(\tau)\theta_{4}(\tau)\,\label{relation:N4r22}.
\end{equation}
This implies that 
\begin{equation*}
A_{4}(Q^{2})={1\over 2}(A_{4}(Q)+B_{4}(Q))
\,,B_{4}(Q^{2})=A_{4}(Q)^{1\over 2}B_{4}(Q)^{1\over 2}
\,, C_{4}(Q^{2})={1\over 2}(A_{4}(Q)-B_{4}(Q))\,.
\end{equation*}
By using the definition of $E_{N}$ for $N=4$ in \eqref{defofE} and the relation \eqref{ABCrelation}, we can then express 
$A_{4}(Q^{1\over 2}),B_{4}(Q^{1\over 2}),C_{4}(Q^{1\over 2}),E_{4}(Q^{1\over 2})$ in terms of $A_{4}(Q),B_{4}(Q),C_{4}(Q),E_{4}(Q)$ as follows
\begin{equation}\label{eqisogenyrelation:N=4}
\left\{
\begin{array}{ll}
A_4(Q^{1\over 2})&=A_4(Q)+C_4(Q))\,,\\
B_4(Q^{1\over 2})&=A_4(Q)-C_4(Q))\,,\\
C_4(Q^{1\over 2})&=2A_4^{1\over 2}(Q)C_4^{1\over 2}(Q)\,,\\
E_4(Q^{1\over 2})&=2E_4(Q)-A_4^{2}(Q)-2A_4(Q)C_4(Q)+C_4^{2}(Q)\,.
\end{array}
\right.
\end{equation}

\paragraph{$(N,r)=(3,3)$.}

To derive the relations among $A_3(Q),B_3(Q),C_3(Q),E_3(Q)$ and $A_3(Q^{1\over 3})$, $B_3(Q^{1\over 3})$, $C_3(Q^{1\over 3})$, $E_3(Q^{1\over 3})$,
we use the relation in \cite{Borwein:1991, Borwein:1994, Berndt:1995}
\begin{equation}
A_3(Q^{1\over 3})=A_3(Q)+2C_3(Q)\,,\quad B_3(Q^{1\over 3})=A_3(Q)-C_3(Q)\,.
\end{equation}
From this one can then derive the relation among $E_3(Q^{1\over 3})$ and $A_3(Q), B_3(Q), C_3(Q), E_3(Q)$.

\paragraph{$(N,r)=(2,4)$.}

Using the $\theta$-expressions for $N=2$ case in Appendix \ref{appendixthetaexpansions} we get
(see \cite{Maier:2011})
\begin{equation}
A_{2}^{2}(Q^{1\over 2})=A_{2}^{2}(Q)+3C_{2}^{2}(Q)\,,\quad B_{2}^{2}(Q^{1\over 2})=A_{2}^{2}(Q)-C_{2}^{2}(Q)\,.
\end{equation}
Recall that the modular forms for $\Gamma_{0}(2)$ and $\Gamma_{0}(4)$ have the following relations from \eqref{C2relation}
\begin{equation*}
B_{2}(Q)=B_{4}(Q)\,,\quad C_{2}(Q)=2^{-{1\over 2}}C_{4}(Q^{1\over 2})=(2A_{4}(Q)C_{4}(Q))^{1\over 2}\,.
\end{equation*}
Hence we can use the quadratic relations for modular forms of $\Gamma_{0}(4)$ to derive those for $\Gamma_{0}(2)$.
Iterating the quadratic relations will give the quartic relations.

\begin{rem}
\label{remarkpolynomiality}
We note that for $N=3,4$ cases, the equations in Proposition \ref{propRamanujanforall} only involve positive powers of the generators.
However, when $N=1^{*},2$ this is not the case. 
For later applications in Section \ref{sectionequivalence} we will 
need to find generators for the ring of the quasi-modular forms for $\Gamma(r)$ with $r=2,3,4,6$.
According to \eqref{eqringofmodularformsgammaN}, the generators for the $r=3,4$ cases can be constructed from those for $\Gamma_{0}(3),\Gamma_{0}(4)$.
The $\Gamma(2)$ case can be related to the $\Gamma_{0}(4)$ case by using the 2-isogeny between the two modular groups, and the above results on 2-isogeny on modular forms for $(N,r)=(4,2)$ case.
The $r=6$ case turns out to be reduced to the $r=2,3$ cases. This will be discussed in Section \ref{sectionr=6}.
Therefore, for all cases, the $\partial_{\tau}$ derivatives of the generators will be in the polynomial ring, that is, no negative power will appear upon taking derivatives.
\end{rem}


\section{From WDVV equations to Ramanujan identities}
\label{sectionequivalence}

The genus zero primary potential in the GW theory of the elliptic orbifolds is fully studied in \cite{Satake:2011, Krawitz:2011} using WDVV equations. 
We now discuss the relation between WDVV equations and Ramanujan identities. As a result, we will prove Theorem \ref{thmequivalence} and Theorem \ref{thmprimarygenuszeromodularity}.
Recall that for each case the parameter $q=e^t$ in the GW theory of $\mathcal{X}_{r}$ is related to the parameter $Q=\exp (2\pi i\tau)$
by $q=Q^{1\over r}$.


\subsection{$\mathcal{X}_{2}=\mathbb{P}^1_{2,2,2,2}$}

\subsubsection{WDVV equations for $\mathbb{P}^1_{2,2,2,2}$}
As shown in \cite{Satake:2011}, using Kawasaki's orbifold Riemann-Roch formula \cite{Kawasaki:1979}, the string equation and the divisor equation, it is easy to see that the genus zero primary potential is
\begin{equation}\label{eqg0N2}
F_{0}^{\mathcal{X}_2}=\mathrm{cubic~ terms}+ X(q)t_1t_2t_3t_4 + \frac{Y(q)}{4!} (\sum_{i=1}^4t_i^4)+ \frac{Z(q)}{2!2!} \sum_{1\leq i,j\leq 4,\, i\neq j}(t_i^2+t_j^2)\,,
\end{equation}
where the parameter $t_{i}$ is the coordinate for vector space spanned by the unit ${\bf 1}\in H$ when $i=0$ and that by the twisted sector
$\Delta_{i}$ when $i=1,2,3,4$.
The coefficients $X(q),Y(q)$ and $Z(q)$ are the GW correlation functions
\begin{equation}\label{GW-inv2}
\left\{
\begin{array}{ll}
X(q)&:=\DL \Delta_1,\Delta_2,\Delta_3,\Delta_4\DR_{0,4}(q),\\
Y(q)&:=\DL \Delta_i,\Delta_i,\Delta_i,\Delta_i\DR_{0,4}(q), \quad i\in\{1,2,3,4\}\\
Z(q)&:=\DL \Delta_i,\Delta_i,\Delta_j,\Delta_j\DR_{0,4}(q), \quad \{i,j\}\subset\{1,2,3,4\}.
\end{array}
\right.
\end{equation}
Note that due to symmetry, $Y(q)$ and $Z(q)$ are independent of the choice of $i$ and $j$.

To derive the WDVV equations we need, let us first consider the dual graphs for the boundary cycles in $\overline{\mathcal{M}}_{0,n}$. A vertex represents a genus zero component, an edge connecting two vertices represents a node, and a tail (or half-edge) represents a marked point on the component represented by the corresponding vertex.
The three boundary cycles in $\overline{\mathcal{M}}_{0,4}$ represented by the graphs below are homologous:
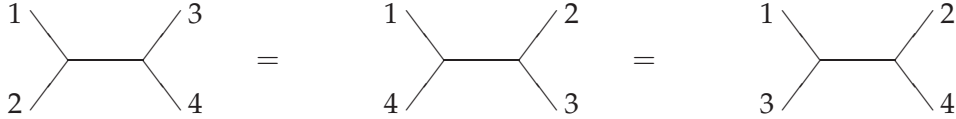
\begin{center}
\begin{picture}(100,20)

	\put(0,9){\line(-3,4){5}}
    \put(0,9){\line(-3,-4){5}}
	\put(0,9){\line(1,0){10}}
	
    \put(10,9){\line(3,4){5}}
    \put(10,9){\line(3,-4){5}}

    \put(-8,2){2}
    \put(-8,14){1}
    \put(16,2){4}
    \put(16,14){3}

\put(25,8){$=$}

\put(75,8){$=$}

	\put(50,9){\line(-3,4){5}}
    \put(50,9){\line(-3,-4){5}}
	\put(50,9){\line(1,0){10}}
	
    \put(60,9){\line(3,4){5}}
    \put(60,9){\line(3,-4){5}}

    \put(42,2){4}
    \put(42,14){1}
    \put(66,2){3}
    \put(66,14){2}

	\put(100,9){\line(-3,4){5}}
    \put(100,9){\line(-3,-4){5}}
	\put(100,9){\line(1,0){10}}
	
    \put(110,9){\line(3,4){5}}
    \put(110,9){\line(3,-4){5}}

    \put(92,2){3}
    \put(92,14){1}
    \put(116,2){4}
    \put(116,14){2}

\end{picture}
\end{center}
The WDVV equations are then obtained by pulling back the cohomological relations corresponding to the above equalities to $\overline{\mathcal{M}}_{0,n}$ by $\pi_{4,n}: \overline{\mathcal{M}}_{0,n}\to\overline{\mathcal{M}}_{0,4}$ which forgets the last $n-4$ marked points (followed by stabilization) and then integrating the GW classes created by inserting some cohomology classes at these marked points.
Let us give one example to show how it works. We integrate the GW class 
$\Lambda_{0,6}^{\mathcal{X}_2}(\Delta_1,\Delta_2,\Delta_4,\Delta_4, \Delta_3, \Delta_4)$ on the boundary strata over $\overline{\mathcal{M}}_{0,6}$, which are the pull-backs of the first homological equivalence relation in $\overline{\mathcal{M}}_{0,4}$ via $\pi_{4,6}: \overline{\mathcal{M}}_{0,6}\to\overline{\mathcal{M}}_{0,4}$.
The integration involving the first cycle gives the non-vanishing terms 
\begin{equation*}
\frac{1}{2}\DL \Delta_1,\Delta_2,\Delta_3,\Delta_4,P\DR_{0,5}+2XY+2XZ\,,
\end{equation*}
The coefficients come from $\DL {\bf 1}, \Delta_4,\Delta_4\DR_{0,3}=1/2$ and $\eta^{(P, {\bf 1})}=1, \eta^{(\Delta_i,\Delta_i)}=2$, where the inverse matrix of the Poincar\'e pairing is denoted by $\eta^{(-,-)}$.
The corresponding decorated dual graphs are 
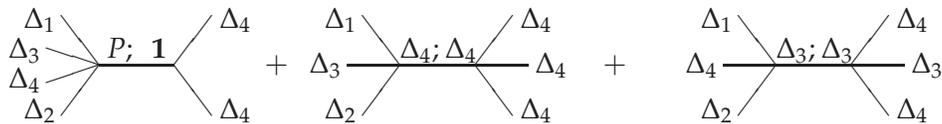
\begin{center}
\begin{picture}(100,20)

	\put(0,9){\line(-3,4){5}}
    \put(0,9){\line(-3,-4){5}}
	\put(0,9){\line(1,0){10}}

    \put(0,9){\line(-3,1){7}}
    \put(0,9){\line(-3,-1){7}}
	
    \put(10,9){\line(3,4){5}}
    \put(10,9){\line(3,-4){5}}

    \put(-10,2){$\Delta_2$}
    \put(-10,14){$\Delta_1$}
    \put(-12,6){$\Delta_4$}
    \put(-12,10){$\Delta_3$}
    \put(16,2){$\Delta_4$}
    \put(16,14){$\Delta_4$}
    \put(1,10){$P;$}
    \put(7,10){$\bf 1$}

    \put(40,10){$\Delta_4;$}
    \put(46,10){$\Delta_4$}

    \put(90,10){$\Delta_3;$}
    \put(96,10){$\Delta_3$}

\put(22,8){$+$}

	\put(40,9){\line(-3,4){5}}
    \put(40,9){\line(-3,-4){5}}
	\put(40,9){\line(1,0){10}}

    \put(40,9){\line(-3,0){7}}
    \put(50,9){\line(3,0){7}}

    \put(50,9){\line(3,4){5}}
    \put(50,9){\line(3,-4){5}}

    \put(30,2){$\Delta_2$}
    \put(30,14){$\Delta_1$}

    \put(56,2){$\Delta_4$}
    \put(56,14){$\Delta_4$}

    \put(28,8){$\Delta_3$}
    \put(58,8){$\Delta_4$}

\put(67,8){$+$}

	\put(90,9){\line(-3,4){5}}
    \put(90,9){\line(-3,-4){5}}
	\put(90,9){\line(1,0){10}}

    \put(90,9){\line(-3,0){7}}
    \put(100,9){\line(3,0){7}}

    \put(100,9){\line(3,4){5}}
    \put(100,9){\line(3,-4){5}}

    \put(80,2){$\Delta_2$}
    \put(80,14){$\Delta_1$}

    \put(106,2){$\Delta_4$}
    \put(106,14){$\Delta_4$}

    \put(78,8){$\Delta_4$}
    \put(108,8){$\Delta_3$}
\end{picture}
\end{center}
We remark that the decorated dual graph where both $\Delta_3$ and $\Delta_4$ appear on the component with two $\Delta_4$-insertions has no contribution, since any term of the form $\DL-,\Delta_3,\Delta_4,\Delta_4,\Delta_4\DR_{0,5}$ vanishes by Kawasaki's Riemann-Roch formula.
Similarly, the integration on the second cycle gives non-vanishing terms
\begin{equation*}
2XZ+2XZ\,,
\end{equation*}
with decorated dual graphs 
\begin{center}
\begin{picture}(150,20)

	\put(40,9){\line(-3,4){5}}
    \put(40,9){\line(-3,-4){5}}
	\put(40,9){\line(1,0){10}}

    \put(40,9){\line(-3,0){7}}
    \put(50,9){\line(3,0){7}}

    \put(50,9){\line(3,4){5}}
    \put(50,9){\line(3,-4){5}}

    \put(30,2){$\Delta_4$}
    \put(30,14){$\Delta_1$}

    \put(56,2){$\Delta_4$}
    \put(56,14){$\Delta_2$}

    \put(28,8){$\Delta_3$}
    \put(58,8){$\Delta_4$}

\put(67,8){$+$}

    \put(40,10){$\Delta_2;$}
    \put(46,10){$\Delta_2$}

    \put(90,10){$\Delta_1;$}
    \put(96,10){$\Delta_1$}

	\put(90,9){\line(-3,4){5}}
    \put(90,9){\line(-3,-4){5}}
	\put(90,9){\line(1,0){10}}

    \put(90,9){\line(-3,0){7}}
    \put(100,9){\line(3,0){7}}

    \put(100,9){\line(3,4){5}}
    \put(100,9){\line(3,-4){5}}

    \put(80,2){$\Delta_4$}
    \put(80,14){$\Delta_1$}

    \put(106,2){$\Delta_4$}
    \put(106,14){$\Delta_2$}

    \put(78,8){$\Delta_4$}
    \put(108,8){$\Delta_3$}
\end{picture}
\end{center}
From the divisor equation \eqref{theta-q}, we then get one \emph{WDVV equation}
\begin{equation*}
\frac{1}{2}\theta_q X+2XY+2XZ=2XZ+2XZ\,.
\end{equation*}
A similar method works for $Y(q)$ and $Z(q)$. More explicitly, we can simplify the WDVV equations and get a system of ODEs
\begin{equation}\label{eqWDVV:N=4}
\left\{
\begin{array}{ll}
\theta_q X&=4X(Z-Y),\\
\theta_q Y&=12Z^2-4X^2-8YZ,\\
\theta_q Z&=4X^2-4Z^2.
\end{array}
\right.
\end{equation}
Furthermore, it is easy to see by direct computation in GW theory (see \cite{Satake:2011}) that
\begin{equation*}
\langle \Delta_1,\Delta_2,\Delta_3,\Delta_4\rangle_{0,4,0}=0\,, \quad \langle \Delta_1,\Delta_2,\Delta_3,\Delta_4\rangle_{0,4,1}=1\,.
\end{equation*}
This implies that the solution to the system of ODEs in \eqref{eqWDVV:N=4} has the following asymptotic behavior
\begin{equation}\label{initial2}
X(q)=q+4q^{3}+\mathcal{O}(q^{5})\,,  \quad Y(q)=-{1\over 4}+\mathcal{O}(q^{3}), \quad Z(q)=2q^{2}+\mathcal{O}(q^{4}).
\end{equation}


\subsubsection{Ramanujan identities for $\Gamma_0(4)$}

For the elliptic orbifold $\mathcal{X}_{2}=\mathbb{P}^{1}_{2,2,2,2}$ 
we have
$Q=\exp(2\pi\sqrt{-1}\tau)=q^{2}$.
This implies in particular that the derivatives $\theta_Q:=Q\partial_Q={1\over 2\pi i}{\partial\over \partial \tau}$  and $\theta_{q}:=q\partial_{q}$ are related by $\theta_{q}=2\theta_{Q}$.

According to the structure theorem in \eqref{eqstructurethm} and the result on $M_{*}(\Gamma_{0}(4),\chi_{-4})$ in \eqref{eqringofmodularformsgamma0N}, we have
$
\widetilde{M}_{*}(\Gamma_{0}(4),\chi_{-4})=\mathbb{C}[A_{4},B_{4}^{2},E_{4}]\,,
$
where $E_{4}$ refers to the quasi-modular form $E_{N}$ for the $N=4$ case.
From Proposition \ref{propRamanujanforall} we can see that the Ramanujan identities for $\Gamma_{0}(4)$ become
\begin{equation}\label{eqRamanujan:N=4}
\left\{
\begin{array}{ll}
\theta_Q A_{4}&={1\over 4} A_{4}(E_{4}+C_{4}^{2}-B_{4}^{2})\,,\\
\theta_Q B_{4}&={1\over 4} B_{4}(E_{4}-A_{4}^{2})\,,\\
\theta_Q C_{4}&={1\over 4} C_{4}(E_{4}+A_{4}^{2})\,,\\
\theta_Q E_{4}&={1\over 4}(E_{4}^{2}-A_{4}^{4})\,.
\end{array}
\right.
\end{equation}
The boundary conditions are as shown in \eqref{boundaryconditionsforGamma0N}.

\subsubsection{Equivalence between WDVV equations and Ramanujan identities}
\label{secequivalence2222}

Using the existence and uniqueness for the solutions to ODEs with boundary conditions, a comparison between the WDVV equations in \eqref{eqWDVV:N=4} and the Ramanujan identities in \eqref{eqRamanujan:N=4}
and their boundary conditions \eqref{initial2}, \eqref{boundaryconditionsforGamma0N} implies the following result.
\begin{lem}
For the correlation functions $X,Y,Z$ of the elliptic orbifold $\mathbb{P}^{1}_{2,2,2,2}$, one has
\begin{equation}\label{thmonequivalence:N=4}
\left\{
\begin{array}{ll}
X(q)&={1\over 16} C^{2}_{4}(q)\,,\\
Y(q)&=-{1\over 16} (3E_{4}(q)+A^2_{4}(q)+C^2_{4}(q))\,,\\
Z(q)&=-{1\over 16} (E_{4}(q)-B^2_{4}(q))\,.
\end{array}
\right.
\end{equation}
\end{lem}
We remark here that the system of WDVV equations given above is a system of first order linear inhomogeneous ODEs with inhomogeneous terms potentially undefined at $q=0$, so the usual existence and uniqueness theorem for the solutions does not hold. Rigorously, we should subtract the first few leading terms in the correlation functions which are not defined when divided by $q$ (coming from the $\theta_{q}$ derivative instead of ordinary derivative $\partial_{q}$). Instead of doing this, we use sufficiently many boundary conditions by taking into consideration of the terms which might give trouble potentially. That is, we give at least the coefficients for the $q^{0},q^{1}$ terms, as we have done in \eqref{initial2}. When dealing with other elliptic orbifold curve cases later in this paper, we shall follow the same procedure.\\

From the result on $\widetilde{M}_{*}(\Gamma_{0}(4),\chi_{-4})$, we know
\begin{equation*}
{1\over 16}C^{2}_{4}(Q)\,,
\quad
-{1\over 16} (3E_{4}(Q)+A^{2}_{4}(Q)+C^{2}_{4}(Q))\,,
\quad
-{1\over 16}( E_{4}(Q)-B^{2}_{4}(Q))
\end{equation*}
are all quasi-modular forms for $\Gamma_{0}(4)$.
Recall that the modular groups $\Gamma_{0}(4)$ and $\Gamma(2)$ are related by the 2-isogeny $\tau\mapsto 2\tau$, so that if $f(Q)$
is a quasi-modular form for $\Gamma_{0}(4)$, then $f(q=Q^{1\over 2})$ is so for $\Gamma(2)$, see e.g. \cite{Zagier:2008}.
Hence we immediately get 
\begin{prop}\label{propbasicgenuszero:N=4}
The system of WDVV equations satisfied by the GW correlation functions $X,Y,Z$ for
$\mathcal{X}_{2}$ is equivalent to the set of Ramanujan identities for the quasi-modular forms $A_{4}^{2},C_{4}^{2},E_{4}$ belonging to the ring
$\widetilde{M}_{\textrm{even}}(\Gamma_{0}(4),\chi_{-4})$. 
Moreover, the GW correlation functions $X, Y$ and $Z$ are weight $2$ quasi-modular forms with respect to the modular group $\Gamma(2)$.
\end{prop}

Here and in the rest of the paper, when we say two sets of equations are "equivalent", we mean that they
are the same up to the action of a linear transformation. This ambiguity will disappear if we replace the system of first-order differential equations by the corresponding higher-order differential equations. 
For example,
in the current case the equations satisfied by $X(q)$ and $C_{4}^{2}(q)/16$ are identical.

For the structure of the correlation functions, we can actually know a little more.
Using \eqref{eqisogenyrelation:N=4}, we arrive at the following consequence.
\begin{cor}
The GW correlation functions $X,Y,Z$ of the elliptic orbifold curve $\mathbb{P}^1_{2,2,2,2}$  belong to the ring $\mathbb{C}[A^2_{4}(Q),A_{4}(Q)C_{4}(Q),C^2_{4}(Q)]$. In fact, we have, under the change of variable $q=Q^{1\over 2}$,
\begin{equation}\label{eqXYZQN=4}
\left\{
\begin{array}{ll}
X(q)&={1\over 4}A_{4}(Q)C_{4}(Q)\,,\\
Y(q)&={1\over 8}(-3E_{4}(Q)+A^{2}_{4}(Q)-2C^{2}_{4}(Q))\,,\\
Z(q)&={1\over 8}(-E_{4}(Q)+A^{2}_{4}(Q))\,.
\end{array}
\right.
\end{equation}
\end{cor}


\subsection{Other cases: $\mathcal{X}_{3},\mathcal{X}_{4},$ and $\mathcal{X}_{6}$}

Now we explain the equivalence between WDVV equations for the genus zero basic correlation functions of $\mathcal{X}_r$ and the Ramanujan identities for $\Gamma(r)$, $r=3,4,6$. For each of them, $\mathcal{X}_r$ is an elliptic orbifold $\mathbb{P}^1$ with three orbifold points. The GW theory in genus zero was studied in \cite{Satake:2011, Krawitz:2011}. Here we follow the setting in \cite{Krawitz:2011}. 

\subsubsection{Reconstruction of genus zero primary correlators from the basic ones}
\label{sectionreconstruction}

Using formula \eqref{dim-virt}, a non-vanishing genus zero primary correlator $\DL\phi_1,\cdots,\phi_n\DR_{0,n}$ must satisfy
\begin{equation*}
\sum_{i=1}^n\deg\phi_i=2n-4\,.
\end{equation*}
Since $0\leq\deg\phi_i\leq2$ and $\deg\phi_i=2$ if and only if $\phi_i=P$,
it is easy to see that the above formula implies there are only finitely many non-vanishing genus zero primary correlators, with no $P$-class as insertions. 
According to the string equation, we have 
\begin{equation*}
\DL {\bf 1}, \phi_1,\cdots,\phi_n\DR_{0,n+1}=0, \quad n\geq3\,.
\end{equation*}
Furthermore, according to the divisor equation \eqref{theta-q}, the correlator with $P$-insertions can be replaced by the $\theta_q$-derivatives of a correlator with no $P$-insertions.
Thus besides $\DL {\bf 1}, \cdots\DR_{0,3}$, all other correlators can be obtained from the correlators with all the insertions being twisted sectors. 
We refer to the twisted sectors $\Delta_i,i=1,2,3$ as \emph{primitive} twisted sectors.
A genus zero primary correlator is called to be \emph{basic} if all of its insertions are twisted sectors and at most two of the insertions are not primitive.

We consider the following form of a WDVV equation, see \cite{Krawitz:2011}, 
\begin{equation}\label{WDVV}
\begin{array}{ll}
\DL\phi_1\bullet\phi_2,\phi_3,\phi_4,\cdots\DR_{0,n+3}=
&-\DL\phi_1,\phi_2,\phi_3\bullet\phi_4,\cdots\DR_{0,n+3}
+\DL\phi_2,\phi_1\bullet\phi_3,\phi_4,\cdots\DR_{0,n+3}\\
&+\DL\phi_1,\phi_3,\phi_2\bullet\phi_4,\cdots\DR_{0,n+3}
+S\,.
\end{array}
\end{equation}
Here $S$ represents all the terms where the number of insertions on each component is at most $n+2$. Besides that, the degree of the first insertion (that is, $\phi_{1}$ or $\phi_2$) on the right hand side of the above equation is strictly smaller than the degree of $\phi_1\bullet\phi_2$. 
By repeating this process, the WDVV equations described above gives an algorithm in obtaining all the genus zero primary correlators from the basic ones.

More explicitly, all the basic correlators for the elliptic orbifold $\mathcal{X}_r$, $r=3,4,6$, and the related WDVV equations were derived and listed in the Appendix in \cite{Shen:2013thesis}. They are already enough to determine all the basic correlators recursively once we know the first few GW invariants in the correlation functions. 

In the rest of this section, we will take the following approach in proving genus zero modularity. We add some WDVV equations to what were listed in \cite{Shen:2013thesis}. Then we prove the basic correlators are quasi-modular forms using the same method we used for the $r=2$ case. After that we use the WDVV equation \eqref{WDVV} to compute all the other genus zero primary correlators, which turn out to be quasi-modular forms as well. 

\subsubsection{$\mathcal{X}_{3}=\mathbb{P}^{1}_{3,3,3}$}
\label{sectiongenuszero:N=3}

Following the notations in \cite{Shen:2013thesis} , c.f. \cite{Krawitz:2011}, up to symmetry, all of the non-vanishing basic correlators of  $\mathbb{P}^{1}_{3,3,3}$ are classified as follows\footnote{Kawasaki's orbifold Riemann--Roch formula is used for the classification.}
\begin{equation}\label{basic-3}
\left\{
\begin{array}{lll}
Z_{1}=\DL \Delta_{1},\Delta_{1}, \Delta_{1} \DR_{0,3} \,,
&Z_{2}=\DL  \Delta_{1}^{2},\Delta_{1},\Delta_{2}^{2}, \Delta_{2}   \DR_{0,4}  \,,
&Z_{3}=\DL  \Delta_{1}^{2},\Delta_{1}^{2},\Delta_{1}, \Delta_{1}   \DR_{0,4} \,,\\
Z_{4}=\DL  \Delta_{1}, \Delta_{2}, \Delta_{3}   \DR_{0,3}\,,
&Z_{5}=\DL  \Delta_{1}^{2},\Delta_{1}^{2},\Delta_{2}, \Delta_{3}    \DR_{0,4}  \,,
&Z_{6}=\DL  \Delta_{1}^{2}, \Delta_{2}^{2},\Delta_{3}, \Delta_{3}    \DR_{0,4} \,.
\end{array}
\right.
\end{equation}
Let us consider the following six WDVV equations, see \cite{Shen:2013thesis},
\begin{equation}\label{eqWDVV:N=3}
\left\{
\begin{array}{ll}
 \theta_q Z_{1}&=9(Z_{4}Z_{6}-Z_{1}Z_{2})\,,\\
\theta_q Z_{3}&=18Z_{2}(Z_{2}-Z_{3})\,,\\
\theta_q Z_{4}&=9Z_{4}(Z_{2}-Z_{3})\,,\\
2Z_{2}Z_{4}&=Z_{1}Z_{5}+Z_{3}Z_{4}\,,\\
Z_{1}Z_{6}&=Z_{4}Z_{5}\,,\\
\theta_{q}^{3}Z_{1}&=27 Z_{6}\theta_{q}Z_{4}-27 Z_{2}\theta_{q}^{2}Z_{1}\,.
\end{array}
\right.
\end{equation}
Not all the WDVV equations involve $\theta_q$ derivatives. For example, the fifth equation above is simplified from the WDVV equation (by taking $\phi_1=\phi_2=\Delta_1, \phi_3=\phi_4=\Delta_2$ in \eqref{WDVV})
\begin{equation*}
\DL \Delta_1\bullet\Delta_1, \Delta_2, \Delta_2, \Delta_3^2\DR_{0,4}+\DL\Delta_1,\Delta_1, \Delta_2\bullet\Delta_2, \Delta_3^2\DR_{0,4}=2\DL \Delta_1\bullet\Delta_2, \Delta_1, \Delta_2,\Delta_3^2\DR_{0,4}\,.
\end{equation*}
A direct computation in GW theory shows 
\begin{equation*}
Z_{1}(q)={1\over 3}+\mathcal{O}(q), \quad Z_{4}(q)=q+\mathcal{O}(q^2)\,.
\end{equation*}
The boundary conditions for \eqref{eqWDVV:N=3} can be obtained by plugging in the above conditions:
\begin{equation}\label{boundary=3}
\left\{
\begin{array}{lll}
   Z_{1}(q)={1\over 3}+\mathcal{O}(q^3)\,, 
& Z_{2}(q)=0+\mathcal{O}(q^3)\,, 
& Z_{3}(q)=-{1\over 9}+\mathcal{O}(q^3)\,,\\
   Z_{4}(q)=q+\mathcal{O}(q^4)\,, 
& Z_{5}(q)={1\over 3}q+\mathcal{O}(q^4)\,,
& Z_{6}(q)=\mathcal{O}(q^2)\,.
\end{array}
\right.
\end{equation}
As mentioned before, these equations with boundaries conditions give a natural algorithm in determining all the basic correlation functions in \eqref{basic-3}. 
Since the system of WDVV equations is over-determined, we are free to add other WDVV equations in determining the correlation functions if necessary.
In the following, we shall use
\begin{equation}\label{N=3extra}
\theta_q Z_{2}=9(Z_{5}Z_{6}-Z_{2}^2)\,.
\end{equation}
Then by straightforward computation, we can show that the ring generated by $Z_{i},i=1,2,\cdots 6$ is closed under the derivative $\theta_{q}$. Indeed, we have\footnote{It would be interesting to explain these relations in a more geometric way.}
\begin{equation}\label{eqnewpolynomialrelations:N=3}
2Z_{2}=Z_{1}^{2}+Z_{3}\,,\quad Z_{5}=Z_{1}Z_{4}\,, \quad Z_{6}=Z_{4}^{2}\,.
\end{equation}
Then we can choose a minimal set of differential equations to be the following ones satisfied by $Z_{1},Z_{4},Z_{3}$:
\begin{equation}\label{eqWDVVsimplified:N=3} 
\left\{
\begin{array}{ll}
& \theta_q Z_{1}={9\over 2}(-Z_{1}Z_{3}+2Z_{4}^{3}-Z_{1}^{3})\,,\\
&\theta_q Z_{4}={9\over 2} Z_{4}(-Z_{3}+Z_{1}^{2})\,,\\
&\theta_q Z_{3}={9\over 2}(-Z_{3}^{2}+Z_{1}^{4})\,.
\end{array}
\right.
\end{equation}
These are identical to the Ramanujan identities satisfied by $A_{3}/3,C_{3}/3,-E_{3}/9$ in Proposition \ref{propRamanujanforall} for the $N=3$ case 
which generate
$
\widetilde{M}_{*}(\Gamma(3))\cong M_{*}(\Gamma(3))\otimes \mathbb{C}[E_{3}]=\mathbb{C}[A_{3},C_{3},E_{3}]$.
By comparing the boundary conditions we get (recall that $q=Q^{1\over 3}$)
\begin{equation}\label{X3-1}
Z_{1}={1\over 3}A_{3}(Q)\,, \quad Z_{4}={1\over 3}C_{3}(Q)\,,\quad Z_{3}=-{1\over 9}E_{3}(Q)\,.
\end{equation}
Plugging these into \eqref{eqnewpolynomialrelations:N=3}, we then obtain
\begin{equation}\label{X3-2}
Z_{2}=-{1\over 18}(E_{3}(Q)-A^{2}_{3}(Q))\,, 
\quad
Z_{5}={1\over 9}A_{3}(Q)C_{3}(Q)\,,\quad Z_{6}={1\over 9}C^{2}_{3}(Q)\,.
\end{equation}
In order to obtain all the primary correlators, we still have to compute the non-basic correlators. As we discussed earlier in Section \ref{sectionreconstruction}, they can be reconstructed by using divisor equation, string equation, and WDVV equations (if no insertion is the unit $\bf{1}$ or the $P$-class). 
After lengthy computations, the genus zero primary potential is shown to be\footnote{These results match those in \cite{Satake:2011}.} 
\begin{eqnarray*}
F_{0}^{\mathcal{X}_3}
=&&\frac{1}{2} t_0^2t
+\frac{1}{3} t_0(t_1t_6+t_2t_5+t_3t_4)
+(t_1t_2t_3) \, {C_{3}\over 3}
+\frac{1}{6} (t_1^3+t_2^3+t_3^3) \, {A_{3}\over 3}
\\
&&+(t_1t_2t_5t_6+t_1t_3t_4t_6+t_2t_3t_4t_5)\, \frac{A^2_{3}-E_{3}}{18}
+\frac{1}{2} (t_1^2t_4t_5+t_2^2t_4t_6+t_3^2t_5t_6) \, {C^2_{3}\over 9}
\\
&&+\frac{1}{2} (t_1t_2t_4^2+t_1t_3t_5^2+t_2t_3t_6^2) \, {A_{3}C_{3}\over 9}
+\frac{1}{4} (t_1^2t_6^2+t_2^2t_5^2+t_3^2t_4^2) \, {-E_{3}\over 9}
\\
&&+\frac{1}{2} (t_1t_4t_5t_6^2+t_2t_4t_5^2t_6+t_3t_4^2t_5t_6) \, {A_{3}C^2_{3}\over 27}
+\frac{1}{4} (t_1t_4^2t_5^2+t_2t_4^2t_6^2+t_3t_5^2t_6^2) \, {A^2_{3}C_{3}\over 27}
\\
&&+\frac{1}{6} (t_1t_6(t_4^3+t_5^3)+t_2t_5(t_4^3+t_6^3)+t_3t_4(t_5^3+t_6^3))  \, {C^3_{3}\over 27}
+\frac{1}{24} (t_1t_6^4+t_2t_5^4+t_3t_4^4) \, {A^3_{3}\over 27}
\\
&&+\frac{1}{8}(t_4^2t_5^2t_6^2) \, {2C^4_{3}+A^3_{3}C_{3}\over 81}
+\frac{1}{36} (t_4^3t_5^3+t_4^3t_6^3+t_5^3t_6^3) \, {A_{3}C^3_{3}\over 27}
\\
&&+\frac{1}{24} (t_4t_5t_6^4+t_4t_5^4t_6+t_4^4t_5t_6)  \, {A^2_{3}C^2_{3}\over 27}
+\frac{1}{720}(t_4^6+t_5^6+t_6^6) \, {2A_{3}C^3_{3}-A^4_{3}\over 27}.
\end{eqnarray*}
Here the parameters $\{t_{0},t,t_{1},t_{2},t_{3},t_{4},t_{5},t_{6}\}$ are the coordinates for $H$ with respect to the basis $\{{\bf{1}},P,\Delta_{1},\Delta_{2},\Delta_{3},\Delta_{3}^{2},\Delta_{2}^{2},\Delta_{1}^{2}\}$.
The weight formula in Theorem \ref{thmprimarygenuszeromodularity} for $\mathcal{X}_3$ is a direct consequence of the reconstruction process.
Therefore, we have proved the following result.
\begin{prop}\label{propbasicgenuszero:N=3}
The system of WDVV equations satisfied by the basic GW correlation functions $Z_{1},Z_{4},Z_{3}$ for
$\mathcal{X}_{3}$ is equivalent to the set of Ramanujan identities for the generators of the ring
$\widetilde{M}_{*}(\Gamma(3))$. Moreover, the correlation functions are quasi-modular forms for $\Gamma(3)$ whose weights are given by the $w=T-2$, where $T$ is the number of twisted sectors among the insertions (which is the same as the number of insertions for these basic ones).
\end{prop}


\subsubsection{$\mathcal{X}_{4}=\mathbb{P}^{1}_{4,4,2}$}

For ease of notation, hereafter we shall denote $x=\Delta_{1}, y=\Delta_{2}, z=\Delta_{3}$ for the three primitive twisted sectors.
We also omit the subscripts $0,n$ in the notation $\DL \cdots \DR_{0,n}$ for the correlation functions.
Up to symmetry, all of the genus zero basic correlation functions are listed as follows
$$
\begin{array}{llll}
Z_{1}=\DL  x,x,x^{2}\DR\,,
&Z_{2}=\DL x,x^{3},z,z\DR\,,
&Z_{3}=\DL  x^{2},x^{2},z,z\DR\,,
&Z_{4}=\DL  x,x^{3},y,y^{3}\DR\,,\\
Z_{5}=\DL z,z,z,z\DR\,,
&Z_{6}=\DL  x,x,x^{3},x^{3}\DR\,,
&Z_{7}=\DL  x,y,z\DR\,,
&Z_{8}=\DL x^{2},x^{3},y,z\DR\,,\\
Z_{9}=\DL x^{2},y,y\DR\,,
&Z_{10}=\DL  x^{2},y^{2},z,z\DR\,,
&Z_{11}=\DL x^{3},x^{3},y,y\DR\,,
&Z_{12}=\DL x^{3},y,y^{2},z\DR\,.
\end{array}
$$
A system of the corresponding WDVV equations are given by (see\footnote{We have corrected some typos in the third and seventh equations in the original work.} \cite{Shen:2013thesis})
\begin{eqnarray}
\theta_{q}Z_{1}&=&8 (Z_{7}Z_{12}-Z_{1}Z_{2})\,,\label{eqWDVV1:N=2}\\
\theta_{q}Z_{1}&=&8(2 Z_{7}Z_{12}-Z_{9}Z_{10}-Z_{1}Z_{3})\,,\label{eqWDVV2:N=2}\\
\theta_{q}Z_{1}&=&-16 Z_{1}Z_{4}+8Z_7Z_{12}\,,\label{eqWDVV3:N=2}\\
\theta_{q}Z_{5}&=&16( -Z_{3}Z_{5} +2Z_{3}^{2}+2Z_{10}^{2} )-2\theta_{q}Z_{3}\,,\label{eqWDVV4:N=2}\\
\theta_{q}Z_{6}&=&8 (-2Z_{2}Z_{6}+Z_{2}^{2})\,,\label{eqWDVV5:N=2}\\
\theta_{q}Z_{7}&=&4 (4Z_{2}-Z_{5})Z_{7}\,,\label{eqWDVV6:N=2}\\
\theta_{q}Z_7 &=&16Z_{1}Z_{8}+16Z_{9}Z_{12}-8Z_{2}Z_{7}\,,\label{eqWDVV7:N=2}\\
 \theta_{q}Z_{9}&=& 8(-Z_{2}Z_{9}+Z_{7}Z_{8})\,,\label{eqWDVV8:N=2}\\
0&=& - Z_{1}Z_{10}+Z_{7}Z_{8}+Z_{2}Z_{9}-Z_{3}Z_{9}\,,\label{eqWDVV9:N=2}\\
0&=& -2  Z_{1} Z_{11}+ Z_{7}Z_{8}\,,\label{eqWDVV10:N=2}\\
0&=&-Z_{1}Z_{12}+ Z_{7}Z_{11}-Z_{8}Z_{9}\,,\label{eqWDVV11:N=2}\\
 \theta_{q}^{2} Z_{1}&=&16 ( -Z_{2}\,\theta_{q}Z_{1}   + Z_{12}\,\theta_{q}Z_{7})\,.\label{eqWDVV12:N=2}
\end{eqnarray}
Similarly as before, the corresponding boundary conditions are calculated to be
$$
\begin{array}{llll}
Z_{1}(q)={1\over 4}+\mathcal{O}(q^4)\,, & Z_{2}(q)=\mathcal{O}(q^4)\,, & Z_{3}(q)=\mathcal{O}(q^4)\,,& Z_{4}(q)=\mathcal{O}(q^4)\,,\\
Z_{5}(q)=-{1\over 4}+\mathcal{O}(q^4)\,, & Z_{6}(q)=-{1\over 16}+\mathcal{O}(q^4)\,, & Z_{7}(q)=q+\mathcal{O}(q^5)\,, &
Z_{8}(q)={1\over 4}q+\mathcal{O}(q^5)\,,\\
Z_{9}(q)=\mathcal{O}(q^2)\,, & Z_{10}(q)=\mathcal{O}(q^2)\,, & Z_{11}(q)=\mathcal{O}(q^2)\,, & Z_{12}(q)=\mathcal{O}(q^2)\,.
\end{array}
$$
Computationally, the basic correlation functions listed above can be determined uniquely by solving the coefficients in their series expansions recursively by using these boundary conditions.

Now we aim to find the closed formulas for these correlation functions. As before, we are free to use additional equations to simplify the computations in this over-determined system.
Here we use the following additional WDVV equations 
\begin{eqnarray}
&&\theta_{q}Z_7=16Z_{1}Z_{8}+16Z_{9}Z_{12}-8Z_{3}Z_{7}\label{eqWDVV13:N=2}\,,\\
&&\theta_{q}Z_{10}=8Z_{10}(4Z_{3}-Z_{5})\label{eqWDVV14:N=2}\,,\\
&&\theta_{q}Z_{4}=4Z_2^2+16Z_8Z_{12}-16Z_2Z_4\label{eqWDVV15:N=2}\,,\\
&&Z_{7}^{2}=4Z_{1}Z_{9}\label{eqWDVV16:N=2}.
\end{eqnarray}

Now we prove the equivalence between these WDVV equations and the Ramanujan equations for the generators of the ring of quasi-modular forms for $\Gamma(4)$ by following the same procedure discussed in the $\mathbb{P}^{1}_{3,3,3}$ case. 
From the above WDVV equations, we obtain 
\begin{equation}\label{eqpolynomialrelations:N=2}
\left\{
\begin{array}{lll}
Z_{2}=Z_{3}=2Z_{4}\,,
&Z_{5}=12 Z_{4}-4(Z_{1}^{2}+Z_{9}^{2})\,,
&Z_{6}=2Z_{4}-Z_{1}^{2}-Z_{9}^{2}\,,\\
Z_{8}=Z_{1}Z_{7}\,,
&Z_{10}=2Z_{11}=Z_{7}^{2}\,,
&Z_{12}=Z_{7}Z_{9}\,.
\end{array}
\right.
\end{equation}
with the additional relation $Z_{7}^{2}=4Z_{1}Z_{9}$.
The system of the WDVV equations reduces to
\begin{equation}\label{eqWDVVsimplified:N=2}
\left\{
\begin{array}{ll}
\theta_{q}Z_{1}&= 8(Z_{7}^{2}Z_{9}-2Z_{1}Z_{4})=16(2Z_{1}Z_{9}^{2}-Z_{1}Z_{4}) \,,\\
\theta_{q}Z_{9}&= 8(Z_{7}^{2}Z_{1}-2Z_{4}Z_{9})=16(2Z_{9}Z_{1}^{2}-Z_{9}Z_{4}) \,,\\
\theta_{q}Z_{4}&
=   -16 Z_{4}^{2}+64 Z_{1}^{2}Z_{9}^{2}\,.
\end{array}
\right.
\end{equation}

Recall that the ring $\widetilde{M}_{*}(\Gamma(4))$ of quasi-modular forms for the group $\Gamma(4)$ is given by
\begin{equation*}
\widetilde{M}_{*}(\Gamma(4))\cong M_{*}(\Gamma(4))\otimes \mathbb{C}[E_{4}(Q)]=\mathbb{C}[A_{4}(Q),C_{4}(Q),C_{2}(Q)]
\otimes \mathbb{C}[E_{4}(Q)]\,,
\end{equation*}
where again
$E_{4}(Q)$ refers to the quasi-modular form $E_{N}$ for the $N=4$ case.
Comparing the equations in \eqref{eqWDVVsimplified:N=2} and the corresponding boundary conditions with those
satisfied by the generators $A_{4}(Q),C_{4}(Q),E_{4}(Q)$ and the corresponding boundary conditions given in
Section \ref{sectionRamanujanidentitiesandboundaryconditions}, and using the fact that 
 $\theta_{q}=4\theta_{Q}$,
we find 
\begin{equation}
Z_{1}={1\over 4}A_{4}(Q)\,,\quad
Z_{9}={1\over 4}C_{4}(Q)\,,\quad
Z_{4}={1\over 16}(A_{4}^{2}(Q)-E_{4}(Q))\,.
\end{equation}
The other correlation functions can then be obtained by plugging the above expressions into the formulas in \eqref{eqpolynomialrelations:N=2}.
Summarizing, we arrive at the following conclusion.
\begin{prop}\label{propbasicgenuszero:N=2}
The system of WDVV equations satisfied by the basic GW correlation functions $Z_{1},Z_{9},Z_{4}$ for
$\mathcal{X}_{4}$ is equivalent to the set of Ramanujan identities for the generators $A_{4},C_{4},E_{4}$ belonging to the ring
$\widetilde{M}_{*}(\Gamma(4))$. Moreover, all the basic correlation functions $Z_1,\cdots, Z_{12}$ are quasi-modular forms for $\Gamma(4)$. 
\end{prop}


\subsubsection{$\mathcal{X}_6=\mathbb{P}^{1}_{6,3,2}$}
\label{sectionr=6}

For the elliptic orbifold $\mathcal{X}_6=\mathbb{P}^{1}_{6,3,2}$, again we denote $x=\Delta_{1},y=\Delta_{2},z=\Delta_{3}$.
Similar to the previous cases, we list all the basic correlation functions following the names in \cite{Shen:2013thesis}
$$
\begin{array}{llll}
Z_{1}=\DL  x,x,x^{4}\DR\,,
&Z_{2}=\DL  x,x^{2},x^{3}\DR\,,
&Z_{3}=\DL  y,y,y\DR\,,
&Z_{4}=\DL  x,x^{5},z,z\DR\,,
\\
Z_{5}=\DL x^{2},x^{4},z,z \DR\,,
&Z_{6}=\DL  x^{3},x^{3},z,z\DR\,,
&Z_{7}=\DL  x,x^{5},y,y^{2}\DR\,,
&Z_{8}=\DL  y,y^{2},z,z\DR\,,
\\
Z_{9}=\DL x,x,x^{5},x^{5}\DR\,,
&Z_{10}=\DL  x,y,z\DR\,,
&Z_{11}=\DL  y,y,y^{2},y^{2}\DR\,,
&Z_{12}=\DL  z,z,z,z\DR\,,
\\
Z_{13}=\DL x,y^{2},y^{2},z\DR\,,
&Z_{14}=\DL  x^{2},x^{5},y,z\DR\,,
&Z_{15}=\DL  x^{3},x^{4},y,z\DR\,,
&Z_{16}=\DL  x,x,y^{2}\DR\,,
\\
Z_{17}=\DL x^{2},y,y\DR\,,
&Z_{18}=\DL  x^{2},y^{2},z,z\DR\,,
&Z_{19}=\DL x^{3},x^{5},y,y\DR\,,
&Z_{20}=\DL x^{4},x^{4},y,y\DR\,,
\\
Z_{21}=\DL x,x^{2},z\DR\,,
&Z_{22}=\DL x,x^{3},x^{5},z\DR\,,
&Z_{23}=\DL x,x^{4},x^{4},z\DR\,,
&Z_{24}=\DL x^{3},z,z,z\DR\,,
\\
Z_{25}=\DL x^{3},y,y^{2},z\DR\,, 
&Z_{26}=\DL x,x^{3},y\DR\,,
&Z_{27}=\DL x^{2},x^{2},y\DR\,,
&Z_{28}=\DL x^{4},y,z,z\DR\,,
\\
Z_{29}=\DL x^{4},y,y,y^{2}\DR\,,
&Z_{30}=\DL x,x^{4},x^{5},y\DR\,,
&Z_{31}=\DL x,x^{4},y^{2},z\DR\,,
&Z_{32}=\DL x^{5},y,y,z\DR\,.
\end{array}
$$

Since the number of WDVV equations (see \cite{Shen:2013thesis}) is large, we shall not display them here but leave them to Appendix \ref{appendixcorrelators}. 
The proof for the equivalence between WDVV equations and
Ramanujan identities is straightforward and will be omitted.
Interested readers can check that indeed
the WDVV equations are satisfied by using the explicit expressions for the correlation functions in terms of 
quasi-modular forms.

Now we shall study the modular group for which the correlation functions are quasi-modular forms. 
We can check that all of the basic genus zero correlation functions are polynomials of $Z_{9}=(-1/36)\mathrm{Ei}_{2}(Q)$ and the following correlation functions (for more details, see the discussions in Appendix \ref{appendixcorrelators}),
\begin{equation}\label{eqgenuszeroresultsfor236}
\begin{array}{lll}
Z_{1}={1\over 6}A_{3}(Q)\,,
&Z_{2}={1\over 6}A_{3}(Q^2)\,,
&Z_{16}={1\over 3}C_{3}(Q)\,,\\
Z_{26}={1\over 3}C_{3}(Q^{2})\,,
&Z_{10}={1\over 2}\theta_{({1\over 2}, 0)}(Q)\theta_{({1\over 6},0)}(Q^{3})\,,
&Z_{21}={1\over 4}(\theta_{2}(Q)\theta_{2}(Q^{3}))\,,
\end{array}
\end{equation}
which are the only correlation functions that have modular weight one. Here we have used the fact $q=Q^{1/ 6}$.
Hence it suffices to consider the modularity of these basic correlation functions.

Since $A_{3}(Q),$ $A_{3}(Q^{2})$ 
are known to be generators of ring of modular forms with trivial characters for $\Gamma_{1}(6)$ which is a free $\mathbb{C}$-algebra generated by them \cite{Bannai:2001},
while $C_{3}(Q)={1\over 2}(A_{3}(Q^{1\over 3})-A_{3}(Q))$ is a modular form with trivial character for $\Gamma(3)$ from \eqref{eqringofmodularformsgammaN}, we know all of $Z_{1},Z_{2},Z_{16}$ are modular forms for
$\Gamma_{1}(6)\cap \Gamma(3)$.
In particular, they are modular forms for $\Gamma(6)$. It is also obvious that $Z_{9}$ is a quasi-modular form for $\Gamma(6)$.\\

To find a modular group for which $Z_{26}, Z_{10},Z_{21}$ are modular forms\footnote{This is perhaps well-known to experts. The authors apologize for their ignorance on this point.}, we use the following identities for the modular forms $A_{3},C_{3}$ which are listed in Appendix \ref{appendixthetaexpansions}:
\begin{equation}\label{ACintermsoftheta}
\left\{
\begin{array}{ll}
A_{3}(Q)&=\theta_{2}(Q^{2})\theta_{2}(Q^{6})+\theta_{3}(Q^{2})\theta_{3}(Q^{6})\,,\\
C_{3}(Q)&=
\left(\theta_{(0,0)}(Q^{2})\theta_{({2\over 3},0)}(Q^{6})+\theta_{({1\over 2},0)}(Q^{2})\theta_{({1\over 6},0)}(Q^{6})\right)\,.
\end{array}
\right.
\end{equation}
Now we define
\begin{equation}
f(Q)=\theta_{2}(Q^{2})\theta_{2}(Q^{6})\,,\quad
g(Q)=\theta_{3}(Q^{2})\theta_{3}(Q^{6})\,,\quad
h(Q)=\theta_{4}(Q^{2})\theta_{4}(Q^{6})\,.
\end{equation}
They satisfy the relations \cite{Borwein:1994}
\begin{eqnarray}
g(Q)&=&f(Q)+h(Q)\,,\\
f(Q)&=& {2\over 3}A(Q)-{2\over 3}A(Q^{4})\,,\label{identityg}\\
g(Q)&=& {1\over 3}A(Q)+{2\over 3}A(Q^{4})\,.
\end{eqnarray}
Furthermore, we can show straightforwardly the following identities
\begin{equation}\label{newthetaidentites}
\left\{
\begin{array}{ll}
\theta_{(0,0)}(Q^{2})\theta_{({2\over 3},0)}(Q^{6})={1\over 2}(g(Q^{1\over 3})-g(Q))\,,\\
\theta_{({1\over 2},0)}(Q^{2})\theta_{({1\over 6},0)}(Q^{6})={1\over 2}(f(Q^{1\over 3})-f(Q))\,.
\end{array}
\right.
\end{equation}
Actually, one only needs to show one of them since that the sum of the left hand sides of these two is $C_{3}(Q)={1\over 2}(A_{3}(Q^{1\over 3})-A_{3}(Q))$
and one has $A_{3}(Q)=f(Q)+g(Q)$ by using \eqref{ACintermsoftheta}.

We then express the generating functions $Z_{10},Z_{21}$ in terms of $f,g$ as follows. By using \eqref{identityg}, we get
\begin{equation}
Z_{21}={1\over 4}f(Q^{1\over 2})={1\over 6}(A_{3}(Q^{1\over 2})-A_{3}(Q^{2}))\,.
\end{equation}
From \eqref{newthetaidentites}, we obtain
\begin{equation}
Z_{10}={1\over 4}(f(Q^{1\over 6})-f(Q^{1\over 2})\,.
\end{equation}
Plugging \eqref{identityg} into the above equation and using $C_{3}(Q)={1\over 2}(A_{3}(Q^{1\over 3})-A_{3}(Q))$, we get
\begin{equation}
Z_{10}={1\over 6}(A_{3}(Q^{1\over 6})-A_{3}(Q^{2\over 3})-A_{3}(Q^{1\over 2})+A_{3}(Q^{2}))
={1\over 3}(C_{3}(Q^{1\over 2})-C_{3}(Q^{2}))\,.\\
\end{equation}

We now consider the correlation functions $Z_{26},Z_{21},Z_{10}$ one by one.
We shall use the results in Section \ref{sectionringofquasimodularforms} and the so-called modular
machine (see for instance \cite{Rankin:1977ab}):\\

\textit{
If $f(\tau)$ is a modular form with character $\chi$ for $\Gamma_{0}(N)$, then
$f(\tau),f(M\tau)$ are so for $\Gamma_{0}(MN)$.}\\

For example, from $A_{3}(Q) \in M_{1}(\Gamma_{0}(3),\chi_{-3})$, one knows that $A_{3}(Q^{2})\in M_{1}(\Gamma_{0}(6),\chi_{-3})$ by applying the modular machine,  see e.g., \cite{Bannai:2001}.
From $C_{3}(Q^{3})\in M_{1}(\Gamma_{0}(9),\chi_{-3})$ (see e.g., \cite{Maier:2011}), and the fact that $\Gamma_{0}(9)$ is related to $\Gamma(3)$ by the $3$-isogeny $\tau\mapsto 3\tau$, one knows that
$C_{3}(Q)\in M_{1}(\Gamma(3))$, as already pointed out in Section \ref{sectionringofquasimodularforms}.

Now we look at $C_{3}(Q^{1\over 2})$ and $C_{3}(Q^{2})$. From the fact that $C_{3}(Q^{3})\in M_{1}(\Gamma_{0}(9),\chi_{-3})$, it follows from the modular machine that
$C_{3}(Q^{3}), C_{3}(Q^{12})\in M_{1}(\Gamma_{0}(36),\chi_{-3})$. Now since $\Gamma_{0}(36)$ is related to $\Gamma(6)$ by the $6$-isogeny $\tau\mapsto 6\tau$, we know 
$C_{3}(Q^{1\over 2}), C_{3}(Q^{2})\in M_{1}(\Gamma(6))$.

Now we consider $A_{3}(Q^{1\over 2})$. By using the modular machine, we can easily see that $A_{3}(Q)\in M_{1}(\Gamma_{0}(12),\chi_{-3})$.
Since $\Gamma_{0}(12)$ is related to $\Gamma_{0}(3)\cap \Gamma(2)$ by conjugacy $\tau\mapsto 2\tau$, we see that $A_{3}(Q^{1\over 2})\in M_{1}(\Gamma_{0}(3)\cap \Gamma(2),\chi_{-3})$.
In particular, it is a modular form with trivial character for $\Gamma(6)$.

In sum, all of the correlation functions $Z_{1},Z_{2},Z_{16},Z_{26},Z_{10},Z_{21}$ are modular forms for $\Gamma(6)$.
We then get the following result.
\begin{prop}\label{propbasicgenuszero:N=1}
The system of WDVV equations satisfied by the basic GW correlation functions $Z_{1},Z_{2},Z_{10},Z_{16},Z_{21},Z_{26}, Z_{9}$ for
$\mathcal{X}_{6}$ is equivalent to the set of Ramanujan identities for the corresponding generators belonging to the ring
$\widetilde{M}_{*}(\Gamma(6))$. Moreover, all the basic correlation functions $Z_1, \cdots, Z_{32}$ are quasi-modular forms for $\Gamma(6)$.
\end{prop}


\subsection{Genus zero modularity summarized}
\label{sectonmodularitysummarized}

Now it is easy to see that
Theorem \ref{thmequivalence} follows from
Proposition \ref{propbasicgenuszero:N=4},  \ref{propbasicgenuszero:N=3},  \ref{propbasicgenuszero:N=2}, \ref{propbasicgenuszero:N=1}. 

For the genus zero non-basic correlators, we need to compute them according to the approach outlined in Section \ref{sectionreconstruction}.
For the $r=2$ case there was no non-basic correlators, the results for the $r=3$ case were shown at the end of Section \ref{sectiongenuszero:N=3}.
We also computed (details omitted) the non-basic genus zero correlators for $r=4, 6$ cases and checked that indeed all of them are indeed polynomials of the basic ones and the modular weights are as asserted.
This then yields a computational proof for Theorem \ref{thmprimarygenuszeromodularity}.

For the genus zero correlators, we find that the modular groups involved have a pattern as summarized here in Table \ref{tablemodularitypattern}.
\begin{table}[h]
  \caption[Modular groups summarized]{Modular groups summarized}
  \label{tablemodularitypattern}
  \renewcommand{\arraystretch}{1.2} 
\begin{displaymath}
\begin{tabular}{c|cccc}
 		 $ \mathrm{elliptic~orbifold}$& $ \mathbb{P}^{1}_{2,2,2,2}$ &  $ \mathbb{P}^{1}_{3,3,3}$ &  $ \mathbb{P}^{1}_{4,4,2}$ &  $ \mathbb{P}^{1}_{6,3,2}$\\ \hline
   $ \mathrm{modular~group~for~elliptic~curve~family }$  &$ \Gamma_{0}(4)$ &  $\Gamma_{0}(3)$ &  $ \Gamma_{0}(2)$ &  $\Gamma_{0}(1^{*})$\\
   $ \mathrm{order~of~automorphism~group}~\mathbb{Z}_{r}$  &$r=2$ &  $3$ &  $4$ &  $6$\\
  $ \mathrm{modular~group~for~correlators }$ &$ \Gamma(2) $&$\Gamma(3)$& $\Gamma(4)$& $\Gamma(6)$
  \end{tabular}
	\end{displaymath}
\end{table}
Here the elliptic curve families are the ones mentioned in Remark \ref{remellipticorbifoldandsingularity}.
The results on modular groups for the $r=3,4,6$ cases are consistent with the results in \cite{Milanov:2011, Milanov:2012qu, Milanov:2014} which are obtained by using mirror symmetry. 


\section{Tautological relations and higher genus modularity}
\label{sectionhighergenusmodularity}

In this section we shall prove
Theorem \ref{thmhighergenusmodularity} for all the CY orbifolds $\mathcal{X}_{r}, r=1,2,3,4,6$.


\subsection{Quasi-modularity at genus one: Getzler's relation}
At genus one, we shall compute the primary genus one correlator $\DL P\DR_{1,1}^{\mathcal{X}_r}$.  This correlation function is important because it will be the only building block besides the genus zero correlation functions in the process of constructing higher genera correlation functions, as will be explained in Section \ref{sectionregreduction}.

For the elliptic curve $\mathcal{X}_1$, we have that $Q=q=e^t$. According to \cite{Bershadsky:1993cx, Dijkgraaf:1995, Kaneko:1995}
\begin{equation}\label{g=1-elliptic}
\DL P\DR_{1,1}^{\mathcal{X}_1}=-{1\over 24}+\sum_{n=1}^{\infty}\sum_{d: d|n} d \exp (2\pi i n\tau)=-{1\over 24}\mathrm{Ei}_2(Q).
\end{equation}

For the elliptic orbifold $\mathbb{P}^1$'s, our tool to compute $\DL P\DR_{1,1}^{\mathcal{X}_r}$ is Getzler's relation given in \cite{Getzler:1997}.
It is a linear relation among certain codimension-two classes in $H_*(\overline{\mathcal{M}}_{1,4},\mathbb{Q})$. 
We first review some notations.
The dual graph
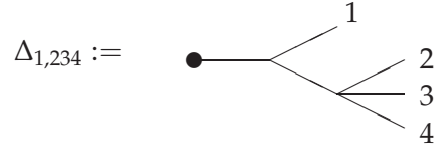
\begin{center}
\begin{picture}(50,20)
    \put(-14,9){$\Delta_{1,234}:=$}

    \put(30,14){1}

    \put(40,8){2}
    \put(40,3){3}
    \put(40,-2){4}

    \put(10,9){\circle*{2}}


	\put(29,4.5){\line(1,0){9}}
	\put(11,9){\line(1,0){9}}
    \put(20,9){\line(2,1){9}}
    \put(20,9){\line(2,-1){9}}
    \put(29,4.5){\line(2,1){9}}
    \put(29,4.5){\line(2,-1){9}}
\end{picture}
\end{center}
 represents a codimension-two stratum in $\overline{\mathcal{M}}_{1,4}$: a filled circle represents a genus one component. 
We have the following $S_4$-invariant of the codimension-two stratum in $\overline{\mathcal{M}}_{1,4}$,
\begin{equation*}
\Delta_{1,234}+\Delta_{2,134}+\Delta_{3,124}+\Delta_{4,123}\,,
\end{equation*}
whose corresponding class in $H_4(\overline{\mathcal{M}}_{1,4},\mathbb{Q})$ will be denoted by $\delta_{3,4}$. We also list the unordered dual graphs for other $S_4$-invariant strata below, see \cite{Getzler:1997} for more details. 
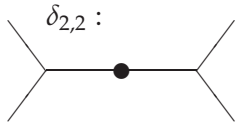
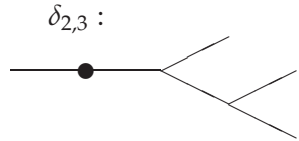
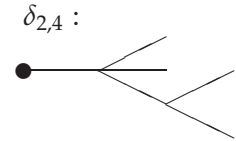
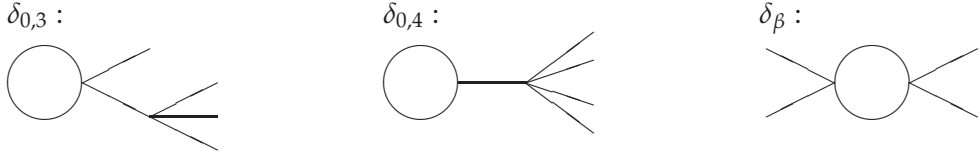
\begin{center}
\begin{picture}(50,20)

    \put(-30,9){\circle*{2}}
	\put(-40,15){$\delta_{2,2}:$}

	\put(-40,9){\line(-3,4){5}}
    \put(-40,9){\line(-3,-4){5}}
	\put(-40,9){\line(1,0){9}}
	\put(-29,9){\line(1,0){9}}
    \put(-20,9){\line(3,4){5}}
    \put(-20,9){\line(3,-4){5}}

    \put(15,9){\circle*{2}}
	\put(10,15){$\delta_{2,3}:$}

	\put(5,9){\line(1,0){9}}
	\put(16,9){\line(1,0){9}}
    \put(25,9){\line(2,1){9}}
    \put(25,9){\line(2,-1){9}}
    \put(34,4.5){\line(2,1){9}}
    \put(34,4.5){\line(2,-1){9}}

    \put(60,9){\circle*{2}}
	\put(60,15){$\delta_{2,4}:$}


	\put(61,9){\line(1,0){9}}
    \put(70,9){\line(2,1){9}}
	\put(70,9){\line(1,0){9}}
    \put(70,9){\line(2,-1){9}}
    \put(79,4.5){\line(2,1){9}}
    \put(79,4.5){\line(2,-1){9}}

\end{picture}
\end{center}

\begin{center}
\begin{picture}(50,20)

	\put(-40,17){$\delta_{0,3}:$}
	\put(10,17){$\delta_{0,4}:$}
	\put(60,17){$\delta_{\beta}:$}

    \put(-35,9){\circle{10}}


	\put(-21,4.5){\line(1,0){9}}
    \put(-30,9){\line(2,1){9}}
    \put(-30,9){\line(2,-1){9}}
    \put(-21,4.5){\line(2,1){9}}
    \put(-21,4.5){\line(2,-1){9}}

    \put(15,9){\circle{10}}


    \put(29,9){\line(4,3){9}}
    \put(29,9){\line(3,1){9}}
	\put(20,9){\line(1,0){9}}
    \put(29,9){\line(4,-3){9}}
    \put(29,9){\line(3,-1){9}}


    \put(75,9){\circle{10}}

    \put(80,9){\line(2,1){9}}
    \put(80,9){\line(2,-1){9}}
    \put(70,9){\line(-2,1){9}}
    \put(70,9){\line(-2,-1){9}}
\end{picture}
\end{center}
In \cite{Getzler:1997}, Getzler found the following identity:
\begin{equation}\label{eq:Getzler}
  12\delta_{2,2}-4\delta_{2,3}-2\delta_{2,4}+6\delta_{3,4}+\delta_{0,3}+\delta_{0,4}-2\delta_{\beta}=0\in H_4(\overline{\mathcal{M}}_{1,4},\mathbb{Q}).
\end{equation}



We now use Getzler's relation \eqref{eq:Getzler} to compute $\DL P\DR_{1,1}^{\mathcal{X}_r}(q)$ in this section. 

\begin{ex}[Genus one correlator for $\mathcal{X}_2$]
For the elliptic orbifold $\mathcal{X}_{2}$, we integrate the cohomology class $\Lambda_2:=\Lambda_{1,4}^{\mathcal{X}_2}(\Delta_1,\Delta_2,\Delta_3,\Delta_4)$ over Getzler's relation \eqref{eq:Getzler}. To do this, first we use Splitting Axiom in GW theory to compute
\begin{equation*}
\int_{\Delta_{1,234}}
\Lambda_2
= \DL P\DR_{1,1}^{\mathcal{X}_2}\,\eta^{(P, {\bf 1})}\DL {\bf 1}, \Delta_1, \Delta_1\DR_{0,3}\,\eta^{(\Delta_1,\Delta_1)}\DL\Delta_1,\Delta_2,\Delta_3,\Delta_4\DR_{0,4}=\DL P\DR_{1,1}^{\mathcal{X}_2}(q)\, X(q)\,.
\end{equation*}
The sum over all strata in $\delta_{3,4}$ gives
\begin{equation*}
\int_{\delta_{3,4}}\Lambda_2
=4\, \DL P\DR_{1,1}^{\mathcal{X}_2}(q)\, X(q)\,.
\end{equation*}
This will be the only term which involves the genus one correlator $\DL P\DR_{1,1}^{\mathcal{X}_r}(q)$ when integrating $\Lambda_2$ over \eqref{eq:Getzler}. 
For each of the strata $\delta_{2,2}, \delta_{2,3}, \delta_{2,4}$, although there is a genus one component, after applying the Splitting Axiom, there must be a zero factor (which comes from a genus zero component) in the integration of $\Lambda_2$. Thus
$$\int_{\delta_{2,2}}\Lambda_2
=\int_{\delta_{2,3}}\Lambda_2
=\int_{\delta_{2,4}}\Lambda_2
=0.$$
Similarly, for the other classes, we obtain
\begin{equation*}
\int_{\delta_{0,3}}\Lambda_2
=X(q)\left(16Y(q)+48Z(q)\right)\,, \quad
\int_{\delta_{0,4}}\Lambda_2
=6\,\theta_q\, X(q)\,, \quad
\int_{\delta_{\beta}}\Lambda_2
=48X(q)Z(q)\,.
\end{equation*}
It is easy to see that Getzler's relation implies
$$\DL P\DR_{1,1}^{\mathcal{X}_2}(q)
={-2Y(q)+6Z(q)\over 3}-{\theta_q X(q)\over 4X(q)}
=\frac{Y(q)}{3}+Z(q)=-{1\over 8}E_{4}(q)+{1\over 24}(B_{4}^{2}(q)-A_{4}^{2}(q))\,.$$
The last equality uses the result in \eqref{thmonequivalence:N=4}. Here no negative powers appear, thanks to Remark \ref{remarkpolynomiality} which implies that for any nonzero monomial $f$ in the ring
$\mathbb{C}[A_{4},B_{4},C_{4}]$ the quantity $\theta_{q}f /f$ still lies in this ring.
Again $\DL P\DR_{1,1}^{\mathcal{X}_2}(q)$ is a quasi-modular form for $\Gamma(2)$ when switching to the variable $Q$,
\begin{equation}\label{g=1,r=2}
\DL P\DR_{1,1}^{\mathcal{X}_2}={-3E_4(Q)+2A_4^2(Q)-C_4^2(Q)\over 12}\,.
\end{equation}
\end{ex}

For the other three elliptic orbifolds $\mathcal{X}_r$, $r=3,4,6$, we integrate $\Lambda_r:=\Lambda_{1,4}^{\mathcal{X}_r}(\Delta_1^2,\Delta_1^{r-1},\Delta_2,\Delta_3)$ over Getzler's relation. According to the choices of the insertions, again we can verify
$$\int_{\delta_{2,2}}\Lambda_r
=\int_{\delta_{2,3}}\Lambda_r
=\int_{\delta_{2,4}}\Lambda_r
=0.$$
The rest of the computation is similar to the $\mathcal{X}_2$ case, although more non-basic correlators will be involved. After tedious computations, we get the following formulas 
\begin{eqnarray}
\DL P\DR_{1,1}^{\mathcal{X}_3}&=&\frac{-2E_3(Q)+A_3^2(Q)}{12}\,,\\
\DL P\DR_{1,1}^{\mathcal{X}_4}&=&\frac{-3E_4(Q)+2A_4^2(Q)-C_4^2(Q)}{12}\,,\\
\DL P\DR_{1,1}^{\mathcal{X}_6}&=& 3Z_{9}=-{1\over 12}\mathrm{Ei}_{2}(Q)\,,
\end{eqnarray}
where $Z_{9}$ is described in the results for correlation functions of $\mathbb{P}^{1}_{6,3,2}$ in Appendix \ref{appendixcorrelators} and $\mathrm{Ei}_{2}(Q)$ is the Eisenstein series.
As a consequence, we arrive at the following conclusion.
\begin{thm}\label{thmgenusonemodularity}
For $r=2,3,4,6$, $\DL P\DR_{1,1}^{\mathcal{X}_r}$ is a weight $2$ quasi-modular form for $\Gamma(r)$. Moreover, we have
\begin{equation}\label{g=1}
\DL P\DR_{1,1}^{\mathcal{X}_r}=-{1\over 12}\mathrm{Ei}_{2}(Q)\,.
\end{equation}
\end{thm}
\begin{proof}
The statement on quasi-modularity and weight follows from the explicit formulas we computed above, and the ring structure of quasi-modular forms.
Equation \eqref{g=1} follows from \eqref{eqGaussSchwarz}, \eqref{eqEintermsofE2} and the following identity (see for example \cite{Maier:2011}) 
\begin{equation}\label{eqC4relation}
C_{4}^{2}(Q)={2\over 3}(-2\mathrm{Ei}_{2}(Q^{4})+3\mathrm{Ei}_{2}(Q^{2})-\mathrm{Ei}_{2}(Q))\,.
\end{equation}
\end{proof}
Let us conclude this section by proving the following formula:
\begin{equation}\label{g=1-derivative}
{\partial\over \partial E_{N}}\DL P\DR_{1,1}^{\mathcal{X}_r}=-{1\over 2r}=-{1\over 2}+{\mu\over 24}\,.
\end{equation}
Here $E_N$ is the generator in \eqref{defofE} and $\mu$ is the rank of Chen-Ruan cohomology $H$ as a graded vector space. 
The equality$ -{1\over 2r}=-{1\over 2}+{\mu\over 24}$ follows by a straightforward computation.
From \eqref{eqGaussSchwarz}, \eqref{eqEintermsofE2} and the identity that $(N+1)r=12$ for $r=3,4$, we obtain
\begin{equation}
\mathrm{Ei}_{2}(Q)={N+1\over 2}(E_{N}(Q)-A_{N}^{2}(Q))+A_{N}^{2}(Q)={6\over r}(E_{N}(Q)-A_{N}^{2}(Q))+A_{N}^{2}(Q)\,.
\end{equation}
For $N=1^{*}, r=6$, we use the convention $E_{N=1^{*}}:=\mathrm{Ei}_{2}(Q)$ (this is essentially the generator $Z_{9}$ for the $r=6$ case, see Appendix \ref{appendixcorrelators}). For $N=4, r=2$, an easy computation using  \eqref{eqGaussSchwarz}, \eqref{eqEintermsofE2}, \eqref{eqC4relation} shows that 
\begin{equation*}
E_{4}(Q)={1\over 3}\mathrm{Ei}_{2}(Q)+{1\over 3}(2A_{4}^{2}(Q)-C_{4}^{2}(Q))\,.
\end{equation*}
Using the above convention for $E_{N}$, we can then see that up to addition by modular forms, the generators $E_{N}$ and $\mathrm{Ei}_{2}$ for the quasi-modular forms are related by
\begin{equation*}
\mathrm{Ei}_{2}(Q)\equiv {6\over r}E_{N}(Q)\,.
\end{equation*}
If we regard a quasi-modular form $f$ as an element in the polynomial ring $\mathbb{C}[A_{N},B_{N},C_{N}][\mathrm{Ei}_{2}]$ or equivalently $\mathbb{C}[A_{N},B_{N},C_{N}][E_{N}]$, then \eqref{g=1-derivative} follows from \eqref{g=1} and the identity
\begin{equation*}
{\partial\over \partial E_{N}}f={6\over r}{\partial \over \partial \mathrm{Ei}_{2}} f\,.
\end{equation*}

We remark that \eqref{g=1-derivative} is the prototype of holomorphic anomaly equation \cite{Bershadsky:1993ta, Bershadsky:1993cx} for the genus one potential. Higher genus potentials also satisfy similar equations. By solving the differential equations recursively, one gets a more efficient way than working with the combinatorics of the moduli spaces of stable maps in determining the GW correlation functions, see for instance \cite{Alim:2013eja, Zhou:2014thesis} for relevant discussions. The holomorphic anomaly equations adapted to the elliptic orbifolds were proposed in \cite{Milanov:2012}. Further discussions will appear elsewhere.


\subsection{Quasi-modularity for all genera: the $g$-reduction}
\label{sectionregreduction}


We recall the \emph{$g$-reduction technique} introduced in \cite{Faber:2010}. It is a consequence of results by Ionel \cite{Ionel:2002} and Faber--Pandharipande \cite{Faber:2005}:

\

\textit{
Let $M(\psi, \kappa)$ be a monomial of $\psi$-classes and $\kappa$-classes with $\deg M\geq g$ for $g\geq1$ or $\deg M\geq1$ for $g=0$, then $M(\psi, \kappa)$ can be represented by a linear combination of dual graphs on the boundary of $\overline{\mathcal{M}}_{g,n}$.}

\

We use the $g$-reduction technique above to prove the quasi-modularity for higher genera correlation functions.
The main result of this section is stated in Theorem \ref{thmhighergenusmodularity} and is recalled below for reference.
\begin{thm}(Theorem \ref{thmhighergenusmodularity})
The ancestor GW correlation functions of the CY 1-fold $\mathcal{X}_r$ are quasi-modular forms for $\Gamma(r)$, $r=1,2,3,4,6$.
The weights are given by 
\begin{equation}\label{weight-g}
w=T+2D+2g-2\,.
\end{equation}
\end{thm}

\begin{proof}
Consider the GW correlation function in \eqref{GW-correlation} given by
$$
\DL\phi_1\psi_1^{k_1},\cdots,\phi_n\psi_n^{k_n}\DR_{g,n}
=\int_{\overline{\mathcal{M}}_{g,n}}\Lambda_{g,n}(\phi_1,\cdots,\phi_n)\,\prod_{i=1}^{n}\psi_i^{k_i}.
$$
According to the dimension formula in \eqref{dim-virt},  the above correlation function vanishes unless
the degree of the insertions is equal to 
2 times the virtual dimension 
of the moduli space,
 that is,
\begin{equation}\label{dim-nonzero}
{1\over 2}\sum_{i=1}^n\deg\phi_{i}+\sum_{i=1}^{n}k_i=(3-\dim_{\mathbb{C}}\mathcal{X}_r)(g-1)+n=2g-2+n\,.
\end{equation}
Since $0\leq \deg\phi_{i}\leq 2$, the above formula implies that if
$
\sum_{i=1}^{n}k_i<2g-2,
$
then the correlator $\DL\phi_1\psi_1^{k_1},\cdots,\phi_n\psi_n^{k_n}\DR_{g,n}$ vanishes.

On the other hand, if 
\begin{equation*}
\deg\left(\prod_{i=1}^{n}\psi_i^{k_i}\right)
:=\sum_{i=1}^{n}k_i\geq
\left\{
\begin{array}{ll}
2g-2,& g\geq1,\\
1,& g=0,
\end{array}
\right.
\end{equation*}
then $\prod_{i=1}^{n}\psi_i^{k_i}$ is a monomial satisfying the condition for the $g$-reduction, thus we can apply this technique to reduce its degree.
Hence by repeatedly using the vanishing condition and the $g$-reduction, it eventually allows us to rewrite any non-vanishing correlation function $\DL\phi_1\psi_1^{k_1},\cdots,\phi_n\psi_n^{k_n}\DR_{g,n}$ as a product of genus zero and genus one primary correlation functions. Moreover, the non-vanishing genus one primary correlation function must be of the form $\DL P,\cdots,P\DR_{1,n}, n\geq1$. 

Therefore it is enough to prove that all the primary correlation functions in genus zero and genus one are quasi-modular forms.
This follows from Theorem \ref{thmprimarygenuszeromodularity}, Theorem \ref{thmgenusonemodularity}, equation \eqref{g=1-elliptic}, and the fact that all the $\theta_q$-derivatives of these quasi-modular forms are polynomials of the generators for the ring of quasi-modular forms (thus are quasi-modular themselves).
In particular, for the elliptic curve case, the basic genus zero correlation functions are constants since only constant maps would contribute. For genus one, as discussed earlier, $\DL P \DR_{1,1}^{\mathcal{X}_1}=-\mathrm{Ei}_{2}(Q)/24$. The Ramanujan identities for the full modular group described in \eqref{eqRamanujanintro} show that the $\theta_Q$-derivatives of $\DL P \DR_{1,1}^{\mathcal{X}_1}$ lie in the ring finitely generated by $\mathrm{Ei}_{2}(Q),\mathrm{Ei}_{4}(Q),\mathrm{Ei}_{6}(Q)$.

The weight formula \eqref{weight-g} holds true for primary genus zero correlation functions and primary genus one correlation function $\DL P,\cdots,P\DR_{1,n}$ by straightforward calculation. 
Since these are the whole building blocks for the $g$-reduction, we obtain the weight formula \eqref{weight-g} by induction.
To be more precise, suppose there is a node that splits the dual graph into two separate components.  Let us denote $T_1, T_2$ to be the number of twisted sectors on the two components, and $D_1, D_2$ be the number of $P$-classes. Then we get
\begin{equation*}
T+2=T_1+T_2, \quad D=D_1+D_2\,,
\quad \mathrm{or}\quad
T=T_1+T_2, \quad D=D_1+D_2-1\,.
\end{equation*}
For the first case, we attach two twisted sectors at the nodes for the two components .
For the second case, we attach the classes ${\bf 1}$ and $P$. In either case, we get 
\begin{equation*}
T+2D=T_1+T_2-2+2D_1+2D_2\,.
\end{equation*}
Thus, using $g=g_1+g_2$, we obtain the following formula from induction:
\begin{equation*}
w=w_1+w_2=(2g_1-2+T_1+2D_1)+(2g_2-2+T_2+2D_2)=2g-2+T+2D\,.
\end{equation*}
A similar argument works for boundary classes where no node splits the dual graph.
\end{proof}

In particular, for the elliptic curve case, an ancestor GW correlation function with only divisor classes and $\psi$-classes as insertions has weight $w=2n+2g-2$.
This is consistent with the result on modular weights for stationary descendant GW correlators studied in \cite{Dijkgraaf:1995, Kaneko:1995, Eskin:2001, Okunkov:2002}.


According to $g$-reduction, an explicit formula representing a monomial $M(\psi, \kappa)$ in terms of boundary cycles will give rise to a formula for the corresponding correlation function in terms of the generators of the corresponding ring of quasi-modular forms. 
For example, we can compute $\DL P\psi_1^2\DR_{2,1}^{\mathcal{X}_{r}}$
using a tautological relation found by Mumford \cite{Mumford:1983}, which says $\psi_1^2$ can be represented by a boundary cycle on $\overline{\mathcal{M}}_{2,1}$. An explicit formula is given in \cite{Faber:thesis, Getzler:1998}. Based on that formula, we obtain:
\begin{equation}\label{eqhighergenusresults}
\DL P\psi_1^2\DR_{2,1}^{\mathcal{X}_{r}}=
{7\over 5}(\DL P\DR_{1,1}^{\mathcal{X}_{r}})^{2}+({1\over 10}+{1\over 10r}+{c_{r}\over 120})\theta_{q}(\DL P\DR_{1,1}^{\mathcal{X}_{r}})\,, 
\end{equation}
where $c_{r}=48,144,252, 480$ for $r=2,3,4,6$.

For higher genus or more insertions, it is still possible to get some closed formulas by the tautological relations found in the literature, for example, in \cite{Getzler:1998, Belorousski:2000, Kimura:2006}, and
in \cite{Pixton:2012, Pandharipande:2015}, which give the most general results. However, the complexity of combinatorics increases very quickly as the genus grows.


\newpage

\begin{appendices}


\section{Modular forms}
\label{appendixthetaexpansions}

We use the convention
$\Theta_{\Lambda}(\tau)=\sum_{x\in \Lambda} \exp ( 2\pi i \cdot {1\over 2}|x|^{2})$ for the theta function of the lattices $\Lambda=E_{8}, D_{4}, A_{2}, A_{1}\oplus A_{1}$, etc. 
We also denote 
\begin{equation}
\theta_{a,b}(z,\tau)=\sum_{n\in \mathbb{Z}} \exp 2\pi i \left({1\over 2}(n+a)^{2}\tau+ (n+a)(z+b)\right) 
\end{equation}
for the theta functions with characters.
We use the following convention for the $\theta$-constants:
\begin{equation}
\theta_{2}(\tau)=\theta_{({1\over 2},0)}(0,\tau),\quad 
\theta_{3}(\tau)=\theta_{(0,0)}(0,\tau),\quad 
\theta_{4}(\tau)=\theta_{(0,{1\over 2})}(0,\tau)\,.
\end{equation}

The expressions for the modular forms $A,B,C$ for the modular group $\Gamma_{0}(N)$ in terms of $\theta$-functions are as follows ($Q=\exp (2\pi i\tau)$)
\begin{eqnarray*}
&N=1^{*}:&\, A(\tau)=\Theta_{E_{8}}^{1\over 4}(\tau)=\mathrm{Ei}_{4}(\tau)^{1\over 4}\,.\\
&N=2:&\nonumber\,\\
&A(\tau)&=\Theta_{D_{4}}^{1\over 2}(\tau) =\left({1\over
4}(\theta_{3}^{4}(\tau)+\theta_{4}^{4}(\tau))^{2}\right)^{1\over 4}
=\left((\theta_{2}^{4}(2\tau)+\theta_{3}^{4}(2\tau))^{2}\right)^{1\over 4}\,,\\
&B(\tau)&=\theta_{3}(\tau)\theta_{4}(\tau)=\theta_{4}^{2}(2\tau)\,,\\
&C(\tau)&=2^{-{1\over 2}}\theta_{2}^{2}(\tau)=2^{1\over 2}
\theta_{2}(2\tau)\theta_{3}(2\tau)\,.
\end{eqnarray*}
\begin{eqnarray*}
&N=3:
&\nonumber\\
& A(\tau)&=\sum_{(m,n)\in
\mathbb{Z}^{2}}Q^{m^{2}+mn+n^2}=\Theta_{A_{2}}(\tau)=
\theta_{2}(2\tau)\theta_{2}(6\tau)+\theta_{3}(2\tau)\theta_{3}(6\tau),\\
&B(\tau)&=\sum_{(m,n)\in \mathbb{Z}^{2}}e^{2\pi i {m-n\over 3}}Q^{m^{2}+mn+n^2}\,,\\
&C(\tau)&=\sum_{(m,n)\in \mathbb{Z}^{2}}Q^{(m+{1\over 3})^{2}+(m+{1\over 3})(n+{1\over 3})+(n+{1\over 3})^2}=\sum_{(m,n)\in \mathbb{Z}^{2}}Q^{m^{2}+mn+n^2+m+n+{1\over 3}}\\
&&={1\over 2}(A({\tau\over 3})-A(\tau))\,,\\
&&=\theta_{0,0}(0,2\tau)\, \theta_{{2\over 3},0}(0,6\tau)\,+
\theta_{{1\over 2},0}(0,2\tau)\,
\theta_{{1\over 6},0}(0,6\tau)\,.\\
&N=4:&\, A(\tau)=\Theta_{A_{1}\oplus
A_{1}}(\tau)=\theta_{3}^{2}(2\tau),\quad
B(\tau)=\theta_{4}^{2}(2\tau),\quad C(\tau)=\theta_{2}^{2}(2\tau)\,.
\end{eqnarray*}

See for instance, \cite{Borwein:1991, Berndt:1995, Mohri:2001zz,
Zagier:2008, Maier:2009, Maier:2011} for details.

\newpage

\section{WDVV equations and correlation functions for $\mathbb{P}^{1}_{6,3,2}$ case}
\label{appendixcorrelators}


In this appendix we display the WDVV equations for the basic correlation functions of the $\mathbb{P}^{1}_{6,3,2}$ case. These are taken from \cite{Shen:2013thesis}.
\begin{eqnarray}
&&\theta_{q}^{2}Z_{1}=12 Z_{31}\theta_{q}Z_{10}+24Z_{23}\theta_{q}Z_{21}-24Z_{4}\theta_{q}Z_{1}\,,\\
&&0=3 Z_{26}\theta_{q}Z_{16}-3Z_{16}\theta_{q}Z_{26}+6 Z_{2}\theta_{q}Z_{1}-6Z_{1}\theta_{q}Z_{2}\,,\label{eqminimalwdvv2}\\
&&\theta_{q}^{2}Z_{3}=24Z_{32}\theta_{q}Z_{10}+12Z_{10}\theta_{q}Z_{32}-24Z_{28}\theta_{q}Z_{17}-12Z_{8}\theta_{q}Z_{3}\,,\\
&& \theta_{q}Z_{1}=6 Z_{10}Z_{31}-12 Z_{1}Z_{4}+12 Z_{21}Z_{23}\,,\\
&& 2Z_{1}Z_{4}-2Z_{1}Z_{5}-Z_{16}Z_{28}+Z_{10}Z_{31}+2Z_{21}Z_{23}=0\,,\\
&&6 Z_{2}Z_{5}-6Z_{2}Z_{6}-2Z_{21}Z_{24}+3Z_{18}Z_{26}+6Z_{21}Z_{22}=0\,,\\
&&\theta_{q}Z_{1}=18Z_{16}Z_{30}-18Z_{1}Z_{7}\,,\label{eqminimalwdvv1}\\
&&\theta_{q}Z_{3}=24Z_{10}Z_{32}-6Z_{3}Z_{8}-12Z_{17}Z_{28}\,,\\
&&\theta_{q}Z_{9}=8Z_{4}^{2}-24Z_{4}Z_{9}+24Z_{22}^{2}\label{eqminimalwdvv7}\,,\\
&&\theta_{q}Z_{10}=18Z_{7}Z_{10}-36Z_{9}Z_{10}+36Z_{21}Z_{30}\,,\label{eqderiveZ7Z91}\\
&&\theta_{q}Z_{10}=6Z_{8}Z_{10}-9Z_{10}Z_{11}+18Z_{25}Z_{26}-18Z_{21}Z_{29}\,,\label{eqminimalwdvv5}\\
&&\theta_{q}Z_{10}=6Z_{8}Z_{10}-4Z_{10}Z_{12}+12Z_{4}Z_{10}+12Z_{21}Z_{28}-12Z_{24}Z_{26}\,,\\
&&\theta_{q}Z_{10}=-18Z_{7}Z_{10}+9Z_{3}Z_{13}+18Z_{17}Z_{31}\,,\label{eqderiveZ7Z92}\\
&&\theta_{q}Z_{10}=-18Z_{7}Z_{10}+36Z_{1}Z_{14}+18Z_{16}Z_{32}-36Z_{21}Z_{30}\,,\\
&& 3Z_{1}Z_{14}-3Z_{2}Z_{15}-Z_{21}Z_{28}+3Z_{21}Z_{30}=0\,,\\
&&\theta_{q}Z_{16}=6Z_{10}Z_{13}-12Z_{4}Z_{16}+12Z_{21}Z_{31}\,,\label{eqsolvefore81}\\
&&\theta_{q}Z_{17}=12 Z_{10}Z_{14}-6Z_{8}Z_{17}-12Z_{27}Z_{28}+12Z_{21}Z_{32}\,,\label{eqsolvefore82}\\
&& 3Z_{10}Z_{13}-6Z_{1}Z_{18}-3Z_{8}Z_{16}+6Z_{4}Z_{16}+6Z_{21}Z_{31}=0\,,\\
&&\theta_{q}Z_{17}=-18Z_{7}Z_{17}+36Z_{2}Z_{19}-36 Z_{27}Z_{30}+12Z_{21}Z_{32}\,,\\
&& 2Z_{1}Z_{19}-2Z_{2}Z_{20}+2Z_{26}Z_{30}-Z_{26}Z_{29}=0\,,\\
&&\theta_{q}Z_{21}=18Z_{27}Z_{31}-18Z_{7}Z_{21}+9Z_{13}Z_{17}\,,\\
&&\theta_{q}Z_{21}=36Z_{2}Z_{22}+12Z_{4}Z_{21}-36Z_{9}Z_{21}\,,\label{eqminimalwdvv6}\\
&& 6Z_{1}Z_{22}-6Z_{2}Z_{23}-3Z_{26}Z_{31}=0\,,\\
&&\theta_{q}Z_{21}=12Z_{5}Z_{21}-12Z_{2}Z_{24}+12Z_{4}Z_{21}+6Z_{10}Z_{18}-4Z_{12}Z_{21}\,,\\
&& -Z_{3}Z_{25}+2Z_{26}Z_{32}+2Z_{10}Z_{19}-2Z_{15}Z_{17}=0\,,\\
&&\theta_{q}Z_{26}=12Z_{10}Z_{22}-12Z_{4}Z_{26}\,,\label{eqminimalwdvv4}\\
&& 2Z_{10}Z_{21}-3Z_{16}Z_{17}+6Z_{2}Z_{26}-6Z_{1}Z_{27}=0\,,\\
&& \theta_{q}Z_{26}=6 Z_{10}Z_{25}-6Z_{8}Z_{26}+12Z_{15}Z_{21}-12Z_{2}Z_{28}\,,\\
&& \theta_{q}Z_{26}=18Z_{7}Z_{26}-9Z_{11}Z_{26}+18Z_{16}Z_{19}-18Z_{2}Z_{29}\,,\\
&&\theta_{q}Z_{26}=-36Z_{9}Z_{26}+18Z_{7}Z_{26}+36Z_{2}Z_{30}\,,\\
&& 2Z_{16}Z_{22}-Z_{13}Z_{26}-2Z_{2}Z_{31}=0\,,\label{eqsolveforZ22}\\
&& Z_{10}Z_{28}+3Z_{15}Z_{26}-3Z_{10}Z_{30}-3Z_{1}Z_{32}=0\,.
\end{eqnarray}

\paragraph{Polynomial relations}
The correlation functions satisfy the polynomial relations derived from the WDVV equations
$$\begin{array}{lll}
Z_{3}=2Z_{1}\,,
&Z_{6}=Z_{4}\,,
&Z_{11}=4Z_{9}\,,\\
Z_{13}=2Z_{14}=2Z_{15}\,,
&Z_{17}=Z_{16}\,,
&Z_{24}=6Z_{22}\,,\\
Z_{27}=0\,,
&Z_{29}=2Z_{30}\,,
&Z_{32}=2Z_{31}\,,\\
Z_{13}=2Z_{1}Z_{10}\,,
&Z_{18}=Z_{10}^{2}\,,
&Z_{19}=2Z_{2}Z_{16}\,,\\
Z_{20}=2Z_{1}Z_{16}\,
&Z_{23}=2Z_{2}Z_{21}\,,
&Z_{25}=Z_{16}Z_{10}\,,\\
Z_{28}=2Z_{10}Z_{21}\,,
&Z_{30}=Z_{16}^{2}/2\,,
&Z_{31}=Z_{16}Z_{21}=Z_{26}Z_{10}\,.
\end{array}
$$
The other relations are found to be
\begin{eqnarray*}
&&Z_{5}=Z_{4}-Z_{21}^{2}\,,\\
&&Z_{7}=Z_{9}+Z_{1}^{2}\,,\\
&&Z_{8}=2Z_{4}+2Z_{21}^{2}\,,\\
&&Z_{12}=3 Z_{4} + 9 Z_{9}\,,\\
&& Z_{22}=(Z_{1}+Z_{2})Z_{21}\,.
\end{eqnarray*}
Therefore, all of the generators are polynomials of $Z_{1},Z_{2},Z_{16},Z_{26},Z_{10},Z_{21}$ and $Z_{4},Z_{9}$, while $Z_{4}$ itself is given by
 \begin{equation*}
Z_{4}
=-{3\over 2} Z_{1}^{2} +3Z_{1}Z_{2}+Z_{21}^{2} + {3\over 2}Z_{9}
={3\over 2} Z_{1}^{2} +3Z_{1}Z_{2}-3Z_{2}^{2} + {3\over 2}Z_{9}\,.
\end{equation*}
\paragraph{Equivalence between WDVV equations and Ramanujan identities and full solutions in terms of quasi-modular forms}
We then take the minimal set of WDVV equations to be
\eqref{eqminimalwdvv1},  \eqref{eqminimalwdvv2}, \eqref{eqsolvefore81}, \eqref{eqminimalwdvv4}, \eqref{eqminimalwdvv5}, \eqref{eqminimalwdvv6}, \eqref{eqminimalwdvv7}. Lengthy computations show that these equations satisfied by the correlation functions coincide with the equations satisfied by 
the quasi-modular forms listed in \eqref{eqgenuszeroresultsfor236} with the same boundary conditions.

\end{appendices}

\providecommand{\bysame}{\leavevmode\hbox to3em{\hrulefill}\thinspace}
\providecommand{\MR}{\relax\ifhmode\unskip\space\fi MR }
\providecommand{\MRhref}[2]{%
  \href{http://www.ams.org/mathscinet-getitem?mr=#1}{#2}
}
\providecommand{\href}[2]{#2}

\bigskip{}

\noindent{\small Department of Mathematics,  Building 380, Stanford, California 94305, USA}

\noindent{\small Email: \tt yfshen@stanford.edu}

\medskip{}
\noindent{\small Perimeter Institute for Theoretical Physics, 31 Caroline Street North, Waterloo, Ontario N2L 2Y5, Canada}

\noindent{\small Email: \tt jzhou@perimeterinstitute.ca}

\end{document}